\documentclass[]{interact}

\usepackage{epstopdf}
\usepackage{subfigure}

\usepackage{natbib}
\bibpunct[, ]{(}{)}{;}{a}{}{,}
\renewcommand\bibfont{\fontsize{10}{12}\selectfont}

\theoremstyle{plain}
\usepackage{bm}
\usepackage{cases}
\usepackage{float}
\usepackage{color}
\usepackage{lipsum}
\usepackage{url}
\usepackage[export]{adjustbox}
\usepackage{subfig}

\usepackage{footnote}
\makesavenoteenv{tabular}
\makesavenoteenv{table}
\usepackage{graphics} 
\usepackage{epsfig} 
\usepackage{array,booktabs,calc}
\usepackage[algo2e]{algorithm2e} 
\usepackage{algpseudocode}  
\usepackage{lipsum}
\usepackage{algorithm}
\usepackage{algpseudocode}
\algnewcommand{\LineComment}[1]{\State \(\triangleright\) #1}
\makeatletter
\def\BState{\State\hskip-\ALG@thistlm}
\makeatother
\usepackage{multicol}
\usepackage{graphicx}

\usepackage{tkz-graph}
\usepackage{pgf, tikz}
\usetikzlibrary{shapes,shapes.geometric,arrows,fit,calc,positioning,automata,}
\tikzset{diamond state/.style={draw,diamond}}
\usepackage{changes}

\newtheorem{definition}{\textbf{Definition}}

\newtheorem{theorem}{\textbf{Theorem}}
\newtheorem{remark}{\textbf{Remark}}

\newtheorem{problem}{\textbf{Problem}}

\begin{document}

\articletype{ARTICLE TEMPLATE}

\title{Incomplete Analytic Hierarchy Process with Minimum \textcolor{black}{Weighted} Ordinal Violations}

\author{
\name{L.~Faramondi\textsuperscript{a}\thanks{CONTACT L. Faramondi. Email: l.faramondi@unicampus.it}, G. Oliva\textsuperscript{a}, and S\'andor~Boz\'oki\textsuperscript{b,c}}
\affil{\textsuperscript{a}Unit of Automatic Control, Department of Engineering, Universit\`a Campus Bio-Medico di Roma, via \`Alvaro del Portillo 21, 00128, Rome, Italy;\\
 \textsuperscript{b} Laboratory on Engineering and Management Intelligence, Research Group of Operations Research and Decision Systems, Institute for Computer Science and Control (SZTAKI), Budapest 1518, P.O. Box 63, Hungary;\\ \textsuperscript{c} Department of Operations Research and Actuarial Sciences, Corvinus University of Budapest, Hungary.}
}

\maketitle

\begin{abstract}
\color{black}
Incomplete pairwise comparison matrices offer a natural way of expressing preferences in decision making processes.
Although ordinal information is crucial, there is a bias in the literature: cardinal models dominate.  Ordinal models usually yield 
non-unique solutions; therefore, an approach blending ordinal and cardinal information is needed. 
In this work, we consider two cascading problems: first, we compute ordinal preferences, maximizing an index that combines ordinal and cardinal information; then, we obtain a cardinal ranking by enforcing ordinal constraints.
Notably, we provide a sufficient condition (that is likely to be satisfied in practical cases) for the first problem to admit a unique solution and we develop a provably polynomial-time algorithm to compute it.
The effectiveness of the proposed method is analyzed and compared with respect to other approaches and criteria at the state of the art.
\color{black}
\end{abstract}

\section{Introduction}
\label{sec:intro}
Pairwise comparisons are applied in several areas among which decision theory and decision support, preference modeling, multi-criteria decision making, voting, ranking, scoring and estimating subjective probabilities of future events.
We focus on multiplicative or reciprocal (\textcolor{black}{$\mathcal{A}_{ij}=1/\mathcal{A}_{ji}$}) pairwise comparison matrices,
where the elements are chosen from a ratio scale, usually composed by the values \mbox{$1/9, 1/8, \ldots, 1/2, 1, 2, \ldots, 8,9$}.
The use of such matrices has {\color{black}become} popular due to the {\em Analytic Hierarchy Process} (AHP) \citep{saaty1977scaling}, see
\citep{GoldenWasilHarker1989,Ho2008,SaatyVargas2012,SubramanianRamanathan2012,VaidyaKumar2006}
for a wide variety of applications. Another important and relevant class of decision problems involves
 incomplete pairwise comparison matrices (e.g., see \textcolor{black}{\citep{Harker1987,FedrizziGiove2007,bozoki2010optimal}}), {\color{black} which allow the absence of ratios among some couples of alternatives. 
 
 In both cases, {\color{black}obtaining a weight vector $\bf w$ from the (\textcolor{black}{incomplete}) pairwise comparison matrix \textcolor{black}{$\mathcal{A}$}} is a fundamental task in the decision making process.
In {\color{black}the literature, approaches able to obtaining a vector of absolute weights from rations matrices are divided in two fundamental classes. } 
The first class includes a set of approaches based on the eigenvalues and the associated eigenvectors of the pairwise comparison matrices. Starting from the preliminary results of Wei \citep{wei1952algebraic}, Saaty \citep{saaty1988analytic} and Cogger \citep{cogger1985eigenweight} propose their approaches based on the principle eigenvector of the pairwise comparison matrix. The main issue of this class of approaches is related to the inconsistency in the filling process of the matrices. An accurate analysis about the data sensitivity problem in AHP is presented in \citep{huang2002enhancement}.
The second class of approaches for the identification of absolute weights involves the methods based on optimization problems. Such problems aim at minimizing a distance function between the entries of the pairwise comparison matrix and the absolute weights.
One of the most common approach in literature is the {\em Direct Least Squares} (DLS) proposed in \citep{chu1979comparison,BarzilaiGolany1990}. The author aims to find a vector of weights in order to minimize the Euclidean distance form the pairwise comparison matrix. The same author proposes a modified version of this approach, the  {\em Weighted Least Squares} (WLS). WLS is a non-linear optimization problem based on the minimization of the $L^2$ distance.
 }
The  {\em Logarithmic Least Squares} (LLS) problem \citep{CrawfordWilliams1985,deGraan1980,deJong1984,bozoki2019logarithmic},
originally defined for complete matrices, is extended to the incomplete case in a natural way: taking only the known elements into consideration \citep{Kwiesielewicz1996,TakedaYu1995}.
The  {\em Incomplete Logarithmic Least Squares} problem has been applied for weighting criteria \citep{BenitezCarpitellaCertaIzquierdo2018} and ranking (tennis players \citep{BozokiCsatoTemesi2016}, chess teams \citep{Csato2013} and Go players \citep{ChaoKouLiPeng2018}).
{\color{black}Other relevant approaches are: \textcolor{black}{the Geometric Mean Method~\citep{Kulakowski}, where the weights are assessed using geometric means and taking into account the lack of some comparison,}  the {\em Fuzzy Programming Method} (FPM) \citep{mikhailov2000fuzzy,buyukozkan2004fuzzy} which transforms the problem to find the vector of weights into a fuzzy programming problem, that can easily be solved as a standard linear program, the {\em Robust Estimation Method} (REM)\citep{lipovetsky2002robust} able to provide solution vectors not prone to influence of possible errors among the elements of a pairwise comparison matrix, the {\em Singular Value Decomposition} \citep{gass2004singular} approach which considers a matrix of shares starting from the pairwise comparison matrix and solves an associated eigenproblem, the {\em Correlation Coefficient Maximization Approach} (CCMA)\citep{wang2007priority} based on two optimization problems, one of which leads to an analytic solution, and the {\em Linear Programming Models} (LPM)\citep{chandran2005linear} based on a linear programming formulation, and finally, Srdjevic \citep{srdjevic2005combining} suggests to combine different prioritization methods for deriving the weights vector.} Finally, it is worth mentioning that there are methods in the literature that aim at reconstructing the missing entries of the pairwise comparison matrix; for instance in \citep{bozoki2010optimal} the missing entries that minimize the dominant eigenvalue are chosen, and then the classical dominant eigenvector criterion is adopted to compute the ranking.

%
%

\subsection{Contribution of the Paper}

{\color{black} The largest part of the previously described approaches for the identification of a weights vector is focused on the minimization of a distance between the ratios, given by the pairwise comparison matrix, and the set of absolute weights. This kind of methods does not guarantee the {\color{black} fulfillment} of constraints about the relative preferences. In more details, these approaches provide a solution able to approximate the ratios but, at the same time, considering any two alternatives, there is no guarantee to respect the ordinal preferences that are encoded by the pairwise comparison matrix entries. In other terms, such approaches implicitly discard the relevance of the ordinal information with the goal to identify a solution which approximates the relative ratios.  In some situation such an assumptions are not acceptable. To this end, }
the models and solutions proposed in our paper consider ordinal information as constraints to the cardinal ordering problem.  In more details, the proposed approach consists of an extension of the LLS problem with a procedure composed by two \textcolor{black}{complementary} steps (optimization problems). 
The first problem \textcolor{black}{aims at maximizing the satisfaction of ordinal constraints, weighting more the satisfaction of constraints corresponding to large cardinal preference values. The second problem aims at finding cardinal preferences with additional constraints that reflect the result of the first step.}

\textcolor{black}{The outline of the paper is as follows.
Notations and preliminaries are given in Section~\ref{sec:pre}. In Section~\ref{sec:llsmov} we propose our approach to solve the \textcolor{black}{incomplete} AHP problem by preserving ordinal constraints. Moreover, we introduce the \textcolor{black}{Weighted} Ordinal Satisfaction Index, this measure captures the inconsistencies due to ordinal violations in the solutions of the sparse AHP problem. 
The proposed method are presented on numerical examples in Section 4 with an accurate comparison with alternative methods in literature. \textcolor{black}{ Finally, Section 5 collects some conclusive remarks and future work directions.}}

\section{Notation and Preliminaries}
\label{sec:pre}
\subsection{General Notation}
We denote vectors via boldface letters, while matrices are shown with uppercase letters. We use $A_{ij}$ to address the $(i,j)$-th entry of a matrix $A$ and $x_i$ for the $i$-th entry of a vector ${\bf x}$.
Moreover, we write ${\bm 1}_n$ and ${\bm 0}_n$ to denote a vector with $n$ components, all equal to one and zero, respectively; similarly, we use $1_{n\times m}$ and $0_{n\times m}$ to denote ${n\times m}$ matrices all equal to one and zero, respectively. We denote by $I_n$ the $n\times n$ identity matrix.
We express by $\exp({\bm x})$ and $\ln({\bm x})$ the component-wise exponentiation or logarithm of the vector ${\bm x}$, i.e., a vector such that $\exp({\bm x})_i=e^{x_i}$ and $\ln({\bm x})_i=\ln(x_i)$, respectively. 
{\color{black}Finally, we adopt the notation $\mbox{sign}(A)$ to denote the entry-wise sign of a matrix $A$, i.e., a matrix $\mbox{sign}(A)$ having $(i,j)$-th entry that corresponds to $(\mbox{sign}(A))_{i j} = \mbox{sign}(A_{i j})$, where $\mbox{sign}(A_{i j})=1$ if $A_{i j}>0$, $\mbox{sign}(A_{i j})=0$ if $A_{i j}=0$ and $\mbox{sign}(A_{i j})=-1$, otherwise. }
\subsection{Graph Theory}
Let $G=\{V,E\}$ be a {\em graph} with $n$ nodes \mbox{$V=\{ v_1, \ldots, v_n \}$} and $e$ edges \mbox{$E\subseteq V\times V\setminus\{(v_i,v_i)\,|\, v_i\in V\}$}, where $(v_i,v_j)\in E$ captures the existence of a link from node $v_i$ to node $v_j$.
A graph is said to be {\it undirected} if $(v_i,v_j)\in {E}$ whenever $(v_j,v_i)\in {E}$, and is said to be {\it directed} otherwise.
\color{black}
In the following, when dealing with undirected graphs, we  represent edges using unordered pairs $\{v_i,v_j\}$ in place of the two directed edges $(v_i,v_j),(v_j,v_i)$.
\color{black}
A graph is {\em connected} if for each pair of nodes $v_i,v_j$ there is a path over $G$ that connects them.
Let the neighborhood $\mathcal{N}_i$ of a node $v_i$ in an undirected graph $G$ be the set of nodes $v_j$ that are connected to $v_i$ via an edge $\{v_i,v_j\}$ in $E$. 
The {\em degree} $d_i$ of a node $v_i$ in an undirected graph $G$ is the number of its incoming edges, i.e., $d_i = |\mathcal{N}_i|$.
\color{black}The {\em degree matrix} $D$ of an undirected graph $G$ is the $n\times n$ diagonal matrix such that $D_{ii}=d_i$.
The {\em adjacency matrix} $\texttt{Adj}$ of a directed or undirected graph $G=\{V,E\}$ with $n$ nodes is the $n\times n$ matrix such that $\texttt{Adj}_{ij}=1$ if $(v_i,v_j)\in E$ and $\texttt{Adj}_{ij}=0$, otherwise.
A well know  property of adjacency matrices is that $\texttt{Adj}^2_{ij}>0$ if and only if there is at least one path (respecting the edge's orientation if the graph is directed) from node $v_i$ to node $v_j$ via an intermediate node $v_k$ (see, for instance,~\citep{godsil2001g}).
The {\em Laplacian matrix} associated to an undirected graph $G$ is the $n\times n$ matrix $L$, having the following structure.
$$
L_{ij}=\begin{cases}
-1& \mbox{ if } \{v_i,v_j\}\in E,\\
d_i,& \mbox{ if } i=j,\\
0,& \mbox{ otherwise}. 
\end{cases}
$$
\color{black}
It is well known that $L$ has an eigenvalue equal to zero, and that, in the case of undirected graphs, the multiplicity of such an eigenvalue corresponds to the number of connected components of $G$ \citep{godsil2001g}. Therefore, the eigenvalue zero has multiplicity one if and only if the graph is connected.
\color{black}
A cycle over a directed graph $G$ is a cyclic sequence of edges $\{(v_1,v_2),(v_2,v_3),\ldots,(v_m,v_1)\}$.
Two cycles are said to be {\em edge-disjoint} if they have no edge in common.
\textcolor{black}{The {\em density} $\rho$ of an undirected graph $G=\{V,E\}$ is defined as
$$
\rho=\dfrac{2|E|}{n(n-1)},
$$
i.e., the ratio between the cardinality $|E|$ of the edge set and $n(n-1)/2$, that is, the cardinality of the edges in a complete graph with $n$ nodes.}

\color{black}
\subsection{Convex Constrained Optimization}
We now review the first-order Karush-Kuhn-Tucker (KKT) necessary and sufficient optimality conditions \citep{Zangwill}.
Note that, in view of the later developments of the paper,  we only review  the  conditions where linear constraints are involved\footnote{
We point out that, in the general case of arbitrary convex constraints, additional {\em constraint qualifications} might be required (e.g., Slater's Condition); moreover, in the case of  nonconvex objective functions or constraints, the KKT conditions hold just as necessary conditions. The interested reader is referred to \cite{Zangwill} (and references therein) for a comprehensive overview of the topic.}.
Let us consider a constrained minimization problem having the following structure:
\begin{equation}
\label{prob:convexproblem}
\begin{aligned}
& \underset{{\bm x}\in \mathbb{R}^{n}}{\min}
& & f({\bm x}) \\
& \text{subject to}
& & g_i({\bm x})\leq 0,& \forall i\in\{1,\ldots,q\} \\
& & &  h_i({\bm x})=0,& \forall i\in\{q+1,\ldots,s\}.
\end{aligned}
\end{equation}
where $f:\mathbb{R}^n\rightarrow \mathbb{R}$ is a convex function and all $g_i:\mathbb{R}^n\rightarrow \mathbb{R}$ and all $h_i:\mathbb{R}^n\rightarrow \mathbb{R}$ are linear.


We now review the KKT first-order necessary and sufficient optimality conditions.
\begin{theorem}[KKT First-order Necessary Conditions]
\label{theoKKT}
Consider a constrained optimization problem as in Eq.~\eqref{prob:convexproblem} and let  the {\em Lagrangian function} be defined as follows:
\begin{equation*}
\mathcal{L}({\bm x},{\bm \zeta})=f({\bm x})+\displaystyle \sum_{i=1}^q \zeta_i \, g_i({\bm x}) + \sum_{i=q+1}^s \zeta_i \, h_i({\bm x})
\end{equation*}
where ${\bm \zeta}=[\zeta_1,\ldots, \zeta_s]^T$ collects the \textcolor{black}{{\em Lagrangian multipliers}.}
A necessary and sufficient condition for a point ${\bm x}^*\in \mathbb{R}^n$ to be a global minimum is that there is ${\bm \zeta}^*\in \mathbb{R}^s$ such that
\begin{enumerate}
\item $\nabla_{{\bm x}} \mathcal{L}({\bm x},{\bm \zeta}) |_{\bm x={\bm x}^*,\, \bm \zeta= {\bm \zeta}^*}=0$; 
\item $\zeta^*_i g_i({\bm x}^*)=0, \quad \forall i=1,\ldots,q$;
\item $g_i({\bm x}^*)\leq 0,\quad \forall i=1,\ldots, q$.
\item $h_i({\bm x}^*)= 0,\quad \forall i=q+1,\ldots, s$.
\item $\zeta^*_{i}\geq 0,\quad \forall i=1,\ldots, q$.
\end{enumerate}
\end{theorem}
\color{black}
\subsection{Incomplete Analytic Hierarchy Process} 
\label{sec:AHP}
In this subsection we review the {\em Analytic Hierarchy Process} (AHP) problem when the available information is incomplete. Specifically, we review the problem and discuss the Logarithmic Least Squares approach for solving it.

Let us consider a set of $n$ alternatives, and suppose that each alternative is characterized by an unknown utility or value $w_i>0$. 
Within the AHP problem, the aim is to compute an estimate of the unknown utilities, based on information on relative preferences. 
In the incomplete information case, we are given a value $\textcolor{black}{\mathcal{A}}_{ij}=\epsilon_{ij}w_i/w_j$ for selected pairs of alternatives $i,j$; such a piece of information corresponds to an estimate of the ratio $w_i/w_j$, where $\epsilon_{ij}>0$ is a multiplicative perturbation that represents the estimation error.
Moreover, for all available $\textcolor{black}{\mathcal{A}}_{ij}$, we assume that $\textcolor{black}{\mathcal{A}}_{ji}=\textcolor{black}{\mathcal{A}}_{ij}^{-1}=\epsilon_{ij}^{-1}w_j/w_i$, i.e., the available terms $\textcolor{black}{\mathcal{A}}_{ij}$ and $\textcolor{black}{\mathcal{A}}_{ji}$ are always consistent and satisfy $\textcolor{black}{\mathcal{A}}_{ij}\textcolor{black}{\mathcal{A}}_{ji}=1$.

We point out that, while traditional AHP approaches \citep{saaty1977scaling,crawford1987geometric,barzilai1987consistent} require knowledge on every pair of alternative, in the partial information setting we are able to estimate the vector ${\bf w}=[w_1,\ldots,w_n]^T$ of the utilities, knowing just a subset of the perturbed ratios.
Specifically, let us consider a graph \mbox{$G=\{V,E\}$} with $|V|=n$ nodes; in this view, each alternative $i$ is associated to a node $v_i\in V$, while the knowledge of $w_{ij}$ corresponds to an edge \mbox{$(v_i,v_j)\in E$}. Clearly, since we assume to know $w_{ji}$ whenever we know $w_{ij}$, the graph $G$ is undirected. 
\textcolor{black}{Let $\mathcal{A}$ be the $n\times n$ matrix collecting the terms $\mathcal{A}_{ij}$, with $\mathcal{A}_{ij}=0$ if   $(v_i,v_j)\not \in E$.}

Notice that, in the AHP literature, there is no universal consent on how to estimate the utilities in the presence of perturbations (see for instance the debate in \citep{dyer1990remarks,saaty1990exposition} for the original AHP problem). This is true also in the incomplete information case, see, for instance, \citep{bozoki2010optimal,oliva2017sparse,ECC2018}.
While the debate is still open, we point out that the logarithmic leasts squares approach appears particularly appealing, since it focuses on error minimization.

For these reasons, in Section~\ref{subsec:log} we review the  {\em \textcolor{black}{Incomplete} Logarithmic Least Squares} (\textcolor{black}{ILLS}) Method \citep{bozoki2010optimal,ECC2018},  which represents an extension of the classical {\em Logarithmic Least Squares}  (LLS) Method developed in \citep{crawford1987geometric,barzilai1987consistent} for solving the AHP problem in the complete information case. {\color{black} Moreover, for the sake of completeness, we summarize the main aspects of 
the \textcolor{black}{Incomplete} Direct Least Squares (Section~\ref{sec:sdls}), the \textcolor{black}{Incomplete} Weighted Least Squares (Section~\ref{sec:swls}), and the \textcolor{black}{Incomplete} Eigenvector Approach (Section~\ref{sec:ev}). These methods are compared with our proposed approach in Section~\ref{sec:compar}.

\subsection{\textcolor{black}{Incomplete} Logarithmic Least Squares (\textcolor{black}{ILLS}) Approach to AHP}
\label{subsec:log}
Within the ILLS algorithm, the aim is to find a logarithmic least-squares approximation $ {\bf w}^*$ to the unknown utility vector ${\bf w}$, i.e., to find the vector that solves
\begin{equation}
\label{eq:LLSM}
{\bf w}^*=\underset{{\bm x}\in\mathbb{R}_+^n}{\arg \min}\left\{\dfrac{1}{2}\sum_{i=1}^n\sum_{j\in \mathcal{N}_i}\left(\ln(\textcolor{black}{\mathcal{A}}_{ij})-\ln\left(\frac{x_i}{x_j}\right)\right)^2\right\}.
\end{equation}
An effective strategy to solve the above problem is to operate the substitution ${\bm y}=\ln({\bm x})$, where $\ln(\cdot)$ is the component-wise logarithm, so that Eq.~\eqref{eq:LLSM} can be rearranged as
\begin{equation}
\label{eq:LLSM2}
\begin{aligned}
 {\bf w}^*&=\exp\left(\underset{{\bm y}\in\mathbb{R}^n}{\arg \min}\left\{\dfrac{1}{2}\sum_{i=1}^n\sum_{j\in \mathcal{N}_i}\left(\ln(\textcolor{black}{\mathcal{A}}_{ij})-y_i + y_j\right)^2\right\}\right),
\end{aligned}
\end{equation}
where $\exp(\cdot)$ is the component-wise exponential.
Let us define
$$
\kappa({\bm y})=\dfrac{1}{2}\sum_{i=1}^n\sum_{j\in \mathcal{N}_i}\left(\ln(\textcolor{black}{\mathcal{A}}_{ij})-y_i + y_j\right)^2;
$$
because of the substitution ${\bm y}=\ln({\bm x})$, the problem becomes convex and unconstrained, and its global minimum is in the form ${\bf w}^*=\exp({\bm y}^*)$, where ${\bm y}^*$ satisfies
$$
\frac{ \partial \kappa({\bm y})}{\partial y_i}\Big|_{{\bm y}={\bm y}^*}=\sum_{j\in \mathcal{N}_i}(\ln(\textcolor{black}{\mathcal{A}}_{ij})-y^*_i+y^*_j)=0,\quad \forall i=1,\ldots,n.
$$
Let us consider the $n\times n$ matrix $P$ such that $P_{ij}=\ln(\textcolor{black}{\mathcal{A}}_{ij})$ if ${\textcolor{black}{\mathcal{A}}}_{ij}>0$ and $P_{ij}=0$, otherwise;
we can express the above conditions in a compact form as
\begin{equation}
\label{eq:implicitlaplacian}
L{\bm y}^*=P{\bm 1}_n,
\end{equation}
where $L$ is the Laplacian matrix associated to the graph $G$.
Notice that, since for hypothesis $G$ is undirected and connected, the Laplacian matrix $L$ has rank $n-1$ \citep{godsil2001g}. 
Therefore, a possible way to calculate a vector ${\bm y}^*$ that satisfies the above equation is to fix one arbitrary component of ${\bm y}^*$ and then solve a reduced size system by simply inverting the resulting nonsingular $(n-1)\times (n-1)$ matrix \citep{bozoki2019logarithmic}.

Vector ${\bm y}^{\ast}$ can also be written as the arithmetic mean of vectors calculated from the spanning trees of the graph of comparisons, corresponding to the incomplete additive pairwise comparison matrix $\ln \textcolor{black}{\mathcal{A}}$ \citep{bozoki2019logarithmic}. 
\color{black}
Finally, it is worth mentioning that, when the graph $G$ is connected, the differential equation
$$
\dot {\bm y}(t)= -L{\bm y}(t)+P{\bm 1}_n
$$
asymptotically converges to $y^{\ast}$ (see \citep{olfati2007consensus}), and represents yet another way to compute it. Notably, the latter approach is typically used by the control system community for {\em formation control} of mobile robots, since the computations are easily implemented in a distributed way and can be performed cooperatively by different mobile robots. Therefore, such an approach appears particularly appealing in a distributed computing setting. 

{\color{black}
\subsection{\textcolor{black}{Incomplete} Direct Least Squares (\textcolor{black}{IDLS})}
\label{sec:sdls}
In this section we illustrate an alternative approach as solution for AHP. Starting from the theory of the DLS method \citep{chu1979comparison,Barzilai1997,BarzilaiGolany1990}, we now summarize the approach applicable in \textcolor{black}{an incomplete information} scenario.
The objective of this method is the minimization of the Euclidean distance between the solution and the distribution of the relative weights in the \textcolor{black}{incomplete} pairwise comparison matrix. That is:
\begin{problem}
\label{prob:SDLS}
Find the vector $\bf w$ that solves
\begin{equation}
\begin{aligned}
&{\min\,}\sum_{i=1}^n \sum_{j = 1}^{n} \mbox{sign}(\textcolor{black}{\mathcal{A}}_{ij}) \bigg( \textcolor{black}{\mathcal{A}}_{ij} - \frac{w_i}{w_j} \bigg)^2&\\
& \text{subject to} &  \\
& \sum_{i=1}^n w_i = 1
\end{aligned}\end{equation}
\end{problem}

Note that solving the \textcolor{black}{Incomplete} Direct Least Squares (\textcolor{black}{IDLS}) is a rather difficult task, since the 
objective function is nonlinear and usually nonconvex; moreover the problem might not admit a unique solution.
Finally, approximation schemes such as the Newton's method may require a good initial point to be successfully applied (see  \citep{bozoki2008solution} and references therein for a more detailed discussion on this issue).

}

{\color{black}
\subsection{\textcolor{black}{Incomplete} Weighted Least Squares (\textcolor{black}{IWLS})}
\label{sec:swls}
Starting from the classic formulation of the WLS \citep{blankmeyer1987approaches}, in this section we summarize the main characteristics of the \textcolor{black}{Incomplete} Weighted Least Squares (\textcolor{black}{IWLS}) which is applicable in \textcolor{black}{an incomplete information} setting. More precisely, a solution for the AHP problem is given by the solution to the following problem:
\begin{problem}
\label{prob:SWLS}
Find the vector $\bf w$ that solves
\begin{equation}
\label{prob:SWLSeq}
\begin{aligned}
&{\min\,}\sum_{i=1}^n \sum_{j = 1}^{n} \mbox{sign}(\textcolor{black}{\mathcal{A}}_{ij}) \bigg(\textcolor{black}{\mathcal{A}}_{ij}  w_{j} - {w_i} \bigg)^2&\\
& \text{subject to} &  \\
& \sum_{i=1}^n w_i = 1
\end{aligned}\end{equation}
\end{problem}
}
{\color{black}
\subsection{Eigenvector Approach (EV)}
\label{sec:ev}
\color{black}This approach \citep{Harker1987} is a generalization of the original eigenvector approach from Saaty.
For notational convenience, we review the approach following the equivalent formalism in \citep{oliva2017sparse} where the matrix involved in the computation is based on the available comparisons, rather than on the missing ones as in the original formulation by \citep{Harker1987}.
\color{black}

Specifically, in the Eigenvector Approach (EV), assuming the underlying graph is connected, the ranking is approximated by the dominant eigenvector of the \textcolor{black}{incomplete} matrix
$$
D^{-1}(\textcolor{black}{\mathcal{A}}-I_n)
$$  
where $D$ is the degree matrix, i.e., a diagonal matrix such that $D_{ii}$ is equal to the degree $d_i$ of node $i$ over $G$ (i.e., the amount of available comparisons involving node $i$).}

\subsection{Evaluation Criteria}
\label{sec:criteria}
{\color{black}
As introduced in Section~\ref{sec:intro}, the main methods for the identification of the weights vector from the the pairwise comparison matrices, disagree on the definition of the result, because each method is focused on a different aspect of the problem (although there is recent work in the literature aimed at allowing for tunable performance indices \citep{brunelli2019general}). To this end, with the aim to compare the effectiveness of multiple results from multiple approaches, we summarize the main aspect of the following comparison criteria (the interested reader is referred to \citep{brunelli2018survey} for a comprehensive survey on this topic).

\subsubsection{Minimum Violations (MV)}
\label{sec:MV}
 Minimum Violations (MV) \citep[p.~213]{GolanyKress1993} also known as the Number of Judgment Reversals (NJR) in \citep[p.~217]{AbelMikhailovKeane2018} was introduced to check whether relations $\textcolor{black}{\mathcal{A}}_{ij}>1$ and $x_{i} > x_{j}$ are fulfilled together.
Specifically, each pair of alternatives $i,j$ such that $i$ is preferred to $j$  but \textcolor{black}{$\textcolor{black}{\mathcal{A}}_{ij}<1$} contributes with a score equal to one to the MV indicator, while each pair of equally important alternatives $i,j$ such that $\textcolor{black}{\mathcal{A}}_{ij}\neq 1$ (or vice versa) contributes with a score $1/2$ (\textcolor{black}{but it is added twice, once for $i,j$ and once
for $j,i$, so it gives the same penalty in total as
the other one, where $1$ is added once}); in other words, considering a set of $n$ alternatives, the MV index is defined as
 \begin{equation}
 	\label{eq:mv}
	MV = \sum_{i=1}^n \sum_{j=1}^n V_{ij},
 \end{equation} 
 where 
 $V_{ ij }=\begin{cases} 1 \quad if \quad w_i > w_j \quad and \quad \textcolor{black}{\mathcal{A}}_{ij} < 1,\\
  \frac{1}{2} \quad if \quad w_i = w_j \quad and \quad \textcolor{black}{\mathcal{A}}_{ij} \neq 1, \\
   \frac{1}{2} \quad if \quad w_i \neq w_j \quad and \quad \textcolor{black}{\mathcal{A}}_{ij}  = 1,\\
0 \quad otherwise. 
\end{cases}$\\
}
{

\color{black}
In this view, the larger is MV, the larger the number of ordinal violations in the vector of utilites ${\bm w}$. Note that, the approach presented in this paper aims at minimizing this kind of metrics in order to respect the preferences expressed in the pairwise comparison matrices. With the aim to apply such criterion also in the \textcolor{black}{incomplete} context we propose the following modification of Equation~\ref{eq:mv}:
\begin{equation}
 	\label{eq:mvs}
	MVs = \sum_{i=1}^n\sum_{j=1}^n  \mbox{sign}(\textcolor{black}{\mathcal{A}}_{ij})V_{ij}
 \end{equation} 

In this way we avoid to consider ordinal violations due to the absence of preferences in the pairwise comparison matrix.

\subsubsection{Total Deviation (TD)}
\label{sec:TD}
A large number of approaches for the definition of the utility vector ${\bm w}$ is formulated in terms of an optimization problem characterized by the minimization of some distance {\color{black} measure} between the ratios $\frac{w_i}{w_j}$ and the corresponding entries of the pairwise comparison matrix $\textcolor{black}{\mathcal{A}}_{ij}$.
Considering $n$ alternatives, the error between the two measures is defined by Takeda et al. \citep{takeda1987estimating} and is computed as

\begin{equation}
\label{eq:td}
	TD = \sum_{i=1}^{n} \sum_{j=1}^{n} \bigg( \textcolor{black}{\mathcal{A}}_{ij} - \frac{w_i}{w_j} \bigg)^2
\end{equation}
This criterion measures the Euclidean distance between the ratios obtained from the entries of the weights vector and the initial relative measures.
With the aim to apply this criterion also in the \textcolor{black}{incomplete} case, we take into account the distances only if $\textcolor{black}{\mathcal{A}}_{ij}\neq 0$:
\begin{equation}
\label{eq:tds}
	TDs = \sum_{i=1}^{n} \sum_{j=1}^{n} \mbox{sign}(\textcolor{black}{\mathcal{A}}_{ij})\bigg( {\textcolor{black}{\mathcal{A}}}_{ij} - \frac{w_i}{w_j} \bigg)^2
\end{equation}
}

\color{black}
\section{\textcolor{black}{ILLS Problem with Minimum Weighted Ordinal Violations}}
\label{sec:llsmov}
In this section, we develop a novel framework, namely \textcolor{black}{{\em Incomplete Logarithmic Least Squares with Minimum Weighted Ordinal Violations (ILLS-MWOV)}} applicable in both complete and \textcolor{black}{incomplete} settings.
Specifically, let us consider a situation where we are given a possibly incomplete matrix ${\textcolor{black}{\mathcal{A}}}$ for $n$ alternatives, corresponding to a connected undirected graph $G$ with $n$ nodes. 
The proposed framework consists of two \textcolor{black}{complementary} steps: first of all, we find an ordinal ranking \textcolor{black}{taking into account also cardinal information}; then, we find a cardinal ranking that does not violate the ordinal one defined during the first step.

\subsection{\textcolor{black}{Weighted Ordinal Ranking}}
In view of the developments in this  paper, it is convenient to provide the following definitions.
\textcolor{black}{
\begin{definition}[Pairwise ordinal preference]
\label{def:pairwise}
A pairwise ordinal preference for a pair of alternatives $i,j$ is expressed by the pair $x_{ij},x_{ji}\in\{0,1\}$, where
$$
x_{ij}=
\begin{cases}
1, & \mbox{ if } i \mbox{ is preferred to }j;\\
0,& \mbox{ if no choice on the preference of } i \mbox{ over } j \mbox{ is specified}
\end{cases}
$$
and it holds 
\begin{equation}
\label{eq:noinconsist}
x_{ij}+x_{ji}\leq 1.
\end{equation}
\end{definition}
Notice that the condition in Eq.~\eqref{eq:noinconsist} guarantees to avoid inconsistent situations where the $i$-th alternative is preferred to the $j$-th one and the \mbox{$j$-th} one is preferred to the $i$-th one.
Moreover, we point out that Eq.~\eqref{eq:noinconsist}  allows situations where 
$$x_{ij}=0\quad \mbox{ and }\quad x_{ji}=0,$$
i.e., \textcolor{black}{where no preference is specified for the pair $i,j$}.
Notice that, due to the definition of $x_{ij}$ and to Eq.~\eqref{eq:noinconsist}, the variables $x_{ij}$ and $x_{ji}$ can be combined to provide overall information on the preference expressed for the pair $i,j$; in fact, it holds
$$
x_{ij}-x_{ji}=
\begin{cases}
1 & \mbox{ if } i \mbox{ preferred to }j\\
-1 & \mbox{ if } j \mbox{ preferred to }i\\
0& \mbox{ if no preference is specified for the pair } i,j.
\end{cases}
$$
}

Let us now develop a \textcolor{black}{weighted} indicator of ordinal violation that will be the basis for the proposed optimization problem.
Notice that the proposed metric generates a penalty with magnitude equal to \textcolor{black}{$|\ln(\textcolor{black}{\mathcal{A}}_{ij})|$} whenever the variables $x_{ij}, x_{ji}$ are in contrast with the ordinal information encoded in the ratio $\textcolor{black}{\mathcal{A}}_{ij}$; moreover, it considers a reward with magnitude equal to \textcolor{black}{$|\ln(\textcolor{black}{\mathcal{A}}_{ij})|$} whenever the variables $x_{ij}, x_{ji}$ agree with the ordinal information encoded in the ratio $\textcolor{black}{\mathcal{A}}_{ij}$. This penalty/reward scheme fundamentally differs from the MVs approach, \textcolor{black}{in that pairs corresponding to largest ratios correspond to largest rewards/penalties, while in the case of MV the rewards and penalties are independent on the numerical value of the ratios. As a consequence, the proposed optimization formulation will prioritize the satisfaction of ordinal information corresponding to large relative preference values. Moreover, as it will be made clear later in this section, this choice allows for the existence of a unique solution even in the presence of cycles, under mild hypotheses on the available data.}
\color{black}
\begin{definition}[Weighted Ordinal Satisfaction Index]
Suppose that a pairwise ordinal preference, expressed in terms of the pair  $x_{ij},x_{ji}\in\{0,1\}$, is defined for all pairs $i,j$ of alternatives and denote by $\{x_{ij}\}$ the set  collecting all such variables $x_{ij}$.
The weighted ordinal satisfaction index $\sigma$ is defined as
$$\sigma=\sum_{\textcolor{black}{\{v_i,v_j\}\in E}}\ln(\textcolor{black}{\mathcal{A}}_{ij})(x_{ij}-x_{ji}).$$
\end{definition}
Notice that the above index takes into account only those pairs of alternatives for which a pairwise comparison $\textcolor{black}{\mathcal{A}}_{ij}\neq 0$ is available.
Notice further that for each pair $\{v_i,v_j\}$ of alternatives such that $\textcolor{black}{\mathcal{A}}_{ij}\neq 0$ we consider a contribution $\ln(\textcolor{black}{\mathcal{A}}_{ij})(x_{ij}-x_{ji})$, i.e., a reward equal to $|\ln(\textcolor{black}{\mathcal{A}}_{ij})|$ when the ordinal preferences $x_{ij},x_{ji}$ match with the ordinal information encoded in the ratio $\textcolor{black}{\mathcal{A}}_{ij}$ (e.g., $\textcolor{black}{\mathcal{A}}_{ij}$ is above one and $x_{ij}=1$) and a penalty equal to  $-|\ln(\textcolor{black}{\mathcal{A}}_{ij})|$ when $x_{ij},x_{ji}$ are in disagreement with $\textcolor{black}{\mathcal{A}}_{ij}$ (e.g., $\textcolor{black}{\mathcal{A}}_{ij}>1$ and $x_{ji}=1$). Notably, we consider zero reward/penalty when both $x_{ij}=0$ and $x_{ji}=0$.

\color{black}
Notice that, when the weighted ordinal satisfaction index $\sigma$ is used as a guide to choose the variables $x_{ij}$, we assign zero penalty/reward to {\em ties}, i.e., to those pairs $i,j$ such that $\mathcal{A}_{ij}=1$ (i.e., because $\ln(\mathcal{A}_{ij})=\ln(\mathcal{A}_{ji})=0$). To avoid this issue, we now define the following function.
\begin{definition}
[Tie Index]
Suppose that a pairwise ordinal preference, expressed in terms of the pair  $x_{ij},x_{ji}\in\{0,1\}$, is defined for all pairs $i,j$ of alternatives and denote by $\{x_{ij}\}$ the set  collecting all such variables $x_{ij}$.
The tie index $\tau$ is defined as
$$\tau=-\delta\sum_{\{v_i,v_j\}\in E\,| \mathcal{A}_{ij}=1}(x_{ij}+x_{ji})$$
with $\delta>0$.
\end{definition}
The above index assigns a penalty equal to $\delta$ whenever a variable $x_{ij}$ corresponding to a tie is set to one.
\color{black}

\color{black}
Based on the above indices, we now define the following optimization problem.
\begin{problem}
\label{prob:problem1}
Find the set $\{x^*_{ij}\}$ that solves
\begin{equation}
\label{prob:convexproblem0eq}
\begin{aligned}
& \underset{\{x_{ij}\}\,|\, x_{ij}\in\{0,1\}}{\max\,} \sigma \textcolor{black}{+\tau} &\\
& \text{subject to} &  \\
&\begin{cases}
x_{ij}+x_{ji}\leq {\color{black} 1}, &\forall i,j  {\color{black}  \hskip0.4em\relax s.t. \hskip0.4em\relax i\neq j} \\[0.3cm]
x_{ij}\geq x_{ik}x_{kj}, &\forall i,j,k \hskip0.4em\relax  {\color{black}  s.t. \hskip0.4em\relax i\neq j\neq k} \\[0.3cm]     
\end{cases}
\end{aligned}\end{equation}
\end{problem}
The above problem, aims at finding the \textcolor{black}{set of}  pairwise ordinal preferences for all pairs of alternatives (not just for the ones for which $\textcolor{black}{\mathcal{A}}_{ij}\neq 0$) that maximizes the \textcolor{black}{weighted} ordinal satisfaction index $\sigma$ and guarantees transitivity of the ordinal preferences.
Notice that the first constraint is required for $x_{ij},x_{ji}$ to represent a pairwise ordinal preference. \textcolor{black}{This constraint directly derives from the relation discussed in Definition~\ref{def:pairwise} and it is necessary to prevent the case $x_{ij}=x_{ji}=1$ from happening. }
Moreover, the constraint $x_{ij}\geq x_{ik}x_{kj}$ models the requirement that the ordinal ranking encoded by the variables $\{x_{ij}\}$ is transitive. In other words, if the $i$-th alternative is preferred to the $k$-th one and the $k$-th one is preferred to the $j$-th one, then alternative $i$ must be preferred to alternative $j$ \textcolor{black}{(we reiterate that $x_{ij}=0$ does not imply $j$ is preferred to $i$ but represents the situation where the preference of $i$ over  $j$ is not explicitly decided)}.
\color{black}
Notably, for sufficiently small $\delta$ (e.g., for $0<\delta< \min_{i,j\,|\mathcal{A}_{ij}>1}\{\ln(\mathcal{A}_{ij})\}/|E|$) the contribution of $\tau$ to the objective function $\sigma+\tau$ is negligible but the presence of $\tau$ prevents unnecessary ties to be set to one.
\color{black}
Before discussing the second problem, let us rearrange Problem~\ref{prob:problem1}  as an \textcolor{black}{Integer Linear Programming} (ILP) formulation; this is done by transforming each nonlinear constraint \textcolor{black}{in the form} $x_{ij}\geq x_{ik}x_{kj}$ into a set of linear constraints featuring additional Boolean variables $z_{ijk}$ and by suitably expressing $\sigma$ as a linear combination of the variables $x_{ij}$, \textcolor{black}{as shown in Problem~\ref{prob:problem11}. }
\begin{problem}
\label{prob:problem11}
Find the sets $\{x_{ij}\}$ and $\{z_{ijk}\}$ that solve
\begin{equation}
\label{prob:convexproblem0}
\begin{aligned}
& \underset{\{x_{ij}\}\,|\, x_{ij}\in \{0,1\},\, \{z_{ijk}\}\,|\, z_{ijk}\in \{0,1\}}{\max\,}\,\sum_{\textcolor{black}{\{v_i,v_j\}\in E}}\textcolor{black}{\ln(\textcolor{black}{\mathcal{A}}_{ij})(x_{ij}-x_{ji})}\,\textcolor{black}{-\delta\sum_{\{v_i,v_j\}\in E\,| \mathcal{A}_{ij}=1}(x_{ij}+x_{ji})}&\\
& \text{subject to} &  \\
&\begin{cases}
x_{ij}+x_{ji}\leq {\color{black}1}, &\forall i,j  {\color{black}  \hskip0.4em\relax s.t. \hskip0.4em\relax i\neq j} \\[0.3cm]
{\color{black} x_{ij} \geq z_{ijk}}  &\forall i,j,k  {\color{black}  \hskip0.4em\relax s.t. \hskip0.4em\relax i\neq j\neq k} \\[0.3cm]
z_{ijk}\geq x_{ik}+x_{kj}-1,  &\forall i,j,k  {\color{black}  \hskip0.4em\relax s.t. \hskip0.4em\relax i\neq j\neq k} \\[0.3cm]
z_{ijk}\leq x_{ik},  &\forall i,j,k  {\color{black}  \hskip0.4em\relax s.t. \hskip0.4em\relax i\neq j\neq k} \\[0.3cm]
z_{ijk}\leq x_{kj},  &\forall i,j,k  {\color{black}  \hskip0.4em\relax s.t. \hskip0.4em\relax i\neq j\neq k} \\[0.3cm]
\end{cases}
\end{aligned}\end{equation}
\end{problem}

Notice that Problem~\ref{prob:problem1} features $O(n^2)$ variables, while the ILP formulation in Problem~\ref{prob:problem11} requires $O(n^3)$ variables. However, the amount of constraints in the ILP formulation remains $O(n^3)$ in both formulations.
\color{black}
\subsubsection{Uniqueness of solution}
Generally speaking, ordinal problems may have multiple solutions, especially in the presence of cycles\footnote{\color{black}For instance, we have a cycle when $i$ is preferred to $j$, $j$ is preferred to $k$ and $k$ is preferred to $i$. In this case the information available is remarkably inconsistent.\color{black}}. However, we point out that, within Problem~\ref{prob:problem1}, a blend of ordinal and cardinal information is used (i.e., the index $\sigma$ is weighted with the ratios $\textcolor{black}{\mathcal{A}}_{ij}$); under suitable assumptions, this feature allows to guarantee uniqueness of solution, even in the presence of cycles.
\color{black}

In order to characterize a sufficient condition for the existence of a unique solution to Problem~\ref{prob:problem1}, it is convenient to introduce the directed graph $G_d=\{V,E_d\}$, where $V=\{v_1,\ldots,v_n\}$ is the set of alternatives and $E_d$ is the set of directed edges $(v_i,v_j)$ (i.e., from alternative $i$ to alternative $j$) that correspond to the available ratios ${\textcolor{black}{\mathcal{A}}}_{ij}\geq1$. In other words, for each pair of alternatives $i,j$ for which a ratio is available, we select the edge $(v_i,v_j)$ if $\textcolor{black}{\mathcal{A}}_{ij}\geq 1$ and we select the edge $(v_j,v_i)$ if $\textcolor{black}{\mathcal{A}}_{ji}\geq 1$ (if both $\textcolor{black}{\mathcal{A}}_{ij}=1$ and $\textcolor{black}{\mathcal{A}}_{ji}=1$, we add either one of the edges $(v_i,v_j),(v_j,v_i)$ for the pair $i,j$). 
Let us now consider the cycles over $G_d$ and let us give the following definition.
\begin{definition}
Let us consider a cycle $C$ over $G_d$ and let $\textcolor{black}{\mathcal{A}}_{\min}$ be the minimum ratio associated to a link in the cycle\footnote{\color{black}We reiterate that we consider only links in $G_d$, i.e., links with associated weights $\mathcal{A}_{ij}\geq 1$; hence, it holds \color{black}$\mathcal{A}_{\min}\geq 1$.\color{black}}.
\color{black}
The cycle is said to be {\em ambiguous} if the multiplicity \color{black}of the edges in $C$ associated to a ratio equal to ${\textcolor{black}{\mathcal{A}}}_{\min}$ is more than one.
\end{definition} 

In order to show the relation between ambiguous cycles and multiple optimal solutions, let us consider the example reported in Fig.~\eqref{fig:ambiguous} (there are two links with associated weight  $\textcolor{black}{\mathcal{A}}_{\min}=2$, represented by a red dotted line). In this case, an optimal solutions will feature $x_{23}=1$, $x_{45}=1$ and $x_{51}=1$. However, to fulfill the transitivity requirements, there is a need to either set $x_{21}=1$ or $x_{43}=1$, paying a penalty equal to $-\ln(2)$; in both cases, the solution obtained has the same value of $\sigma= \ln(7)+\ln(5)+\ln(3)+\ln(2)-\ln(2)$ and thus this instance has two  optimal solutions.

Let us now provide a sufficient condition for the uniqueness of the optimal solution to Problem~\ref{prob:problem1}.
\begin{theorem}
\label{theo:unique}
Let ${\textcolor{black}{\mathcal{A}}}$ be given and let $G_d=\{V,E_d\}$ be the directed graph obtained by considering only the directed edges corresponding to ratios $\textcolor{black}{\mathcal{A}}_{ij}\geq1$. If $G_d$ has no ambiguous cycle and the cycles are all edge-disjoint then the solution of Problem~\ref{prob:problem1} is unique.
\end{theorem}
\begin{proof}
In order to prove the statement, let us first focus on a single non-ambiguous cycle $C$ with $m-1$ edges and, without loss of generality, let us denote by $(v_{m},v_1)$ the unique link corresponding to the minimum weight in the cycle.
In this case, it can be noted that the unique optimal solution corresponds to setting $x_{i,i+1}=1$ for all $i=1,\ldots,m-1$ and,  to fulfill the transitivity requirements, there is the need to set $x_{1m}=1$ (paying a penalty equal to \textcolor{black}{$\max\{\mathcal{A}_{\min},\delta\}$}) and $x_{i,i+2}=1$ for all $i=1,\ldots,m-2$.
At this point we observe that, when $G_d$ satisfies the assumptions of this theorem, it is possible to select the variables $x_{ij}$ corresponding to each cycle (including the additional ones introduced for transitivity) in the above way, and then the cycles can be conceptually collapsed into a node, thus resulting in an acyclic graph.
To conclude the proof we observe that, if the graph is acyclic, it is sufficient to select $x_{ij}=1$ for all remaining $(v_i,v_j)\in E_d$ such that $\mathcal{A}_{ij}>1$ \textcolor{black}{and $x_{ij}=x_{ji}=0$ for all remaining ties (i.e., for $(v_i,v_j)\in E_d$ such that $\mathcal{A}_{ij}=1$); moreover,} there is the need to set to one all variables $x_{ij}$ that are required to fulfill the transitivity constraints to obtain the unique optimal solution.
This completes our proof.
\end{proof}
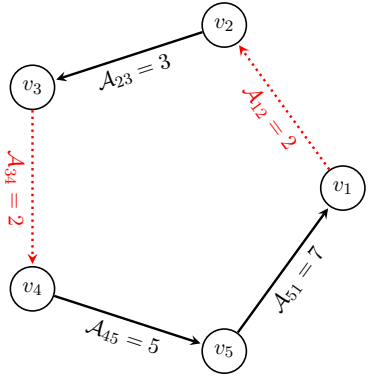
\begin{figure}
\centering
\resizebox{2in}{!} {
\begin{tikzpicture}[
            > = stealth, 
            shorten > = 1pt, 
            auto,
            node distance = 2.5cm, 
            semithick 
        ]

        \tikzstyle{every state}=[
            draw = black,
            thick,
            fill = white,
            minimum size = 4mm
        ]
\color{black}
 	\node[state] (v1) at ({0*72}:3){$v_1$};
        \node[state] (v2) at ({1*72}:3){$v_2$};
        \node[state] (v3)at ({2*72}:3){$v_3$};
        \node[state] (v4) at ({3*72}:3) {$v_4$};
        \node[state] (v5) at ({4*72}:3) {$v_5$};
   
        \path[red,dotted,->,very thick,sloped, below] (v1) edge node {$\mathcal{A}_{12}=2$} (v2);
 \path[black,->,very thick,sloped, below] (v2) edge node {$\mathcal{A}_{23}=3$} (v3);
  \path[red,dotted,->,very thick,sloped, below] (v3) edge node {$\mathcal{A}_{34}=2$} (v4);
   \path[black,->,very thick,sloped, below] (v4) edge node {$\mathcal{A}_{45}=5$} (v5);
    \path[black,->,very thick,sloped, below] (v5) edge node {$\mathcal{A}_{51}=7$} (v1);
    \end{tikzpicture}
    }
    
\caption{\color{black}Example of ambiguous cycle. The problem has two optimal solutions, one featuring $x_{21}=1$ and the other featuring $x_{43}=1$.\color{black}}

 \label{fig:ambiguous}
\end{figure}
\begin{figure}
\centering
\resizebox{2in}{!} {
\color{black}
\begin{tikzpicture}[
            > = stealth, 
            shorten > = 1pt, 
            auto,
            node distance = 2.5cm, 
            semithick 
        ]

        \tikzstyle{every state}=[
            draw = black,
            thick,
            fill = white,
            minimum size = 4mm
        ]

 	\node[state] (v1) at ({1*90}:3){$v_1$};
        \node[state] (v2) at ({0*90}:3){$v_2$};
        \node[state] (v3)at ({2*90}:3){$v_3$};
        \node[state] (v4) at ({3*90}:3){$v_4$};
   
        \path[red,dotted,->,very thick,sloped, below] (v1) edge node {$\mathcal{A}_{12}=2$} (v2);
 \path[->,very thick,sloped, below] (v2) edge node {$\mathcal{A}_{23}=3$} (v3);
  \path[red,dotted,->,very thick,sloped, below] (v3) edge node {$\mathcal{A}_{31}=2$} (v1);
    
     \path[red,dotted,->,very thick,sloped, below] (v3) edge node {$\mathcal{A}_{34}=2$} (v4);
      \path[red,dotted,->,very thick,sloped, below] (v4) edge node {$\mathcal{A}_{42}=2$} (v2);
    \end{tikzpicture}
    }
    \color{black}
\caption{\color{black}Example of instance that does not satisfy the assumptions of Theorem~\ref{theo:unique} (because cycles are ambiguous and not edge-disjoint) but has a unique optimal solution.\color{black}}
 \label{fig:ambiguous2}
\end{figure}
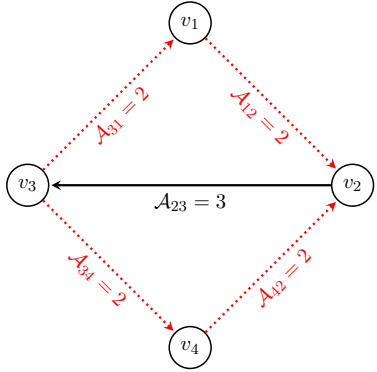

\begin{algorithm}
\color{black}
\caption{\color{black}Solve Problem~\ref{prob:problem1} under the assumptions in Theorem~\ref{theo:unique} \color{black}}\label{algo:alg}
\begin{algorithmic}
\Procedure{$WeightedOrdinalRanking({\textcolor{black}{\mathcal{A}}})$}{}
\State Let $\texttt{Adj}$ be an $n\times n$ adjacency matrix with $\texttt{Adj}_{ij}=0,\,\forall i,j$ 
\State Construct directed graph $G_d=\{V,E_d\}$
\State Find cycles in $G_d$
\State	\For{all cycles $C$ in $G_d$ }{
	\State Find link $(v_i,v_j)$ in $C$ with minimum weight
	\State Set $\texttt{Adj}_{ji}=1$
	}
\State Set $\texttt{Adj}_{ij}=1$ for all remaining $(v_i,v_j)\in E_d$	such that $\mathcal{A}_{ij}>1$
\State	\For{$i=1,\ldots,n-1$}{
	\State $\texttt{Adj}= \mbox{sign}(\texttt{Adj}+\texttt{Adj}^2)$
	}
	\State Set $X^*=\texttt{Adj}$
	\State \Return $X^*$
\EndProcedure
\end{algorithmic}
\color{black}
\end{algorithm}

A few remarks are now in order.
\begin{remark}
The proof of Theorem~\ref{theo:unique} is constructive, i.e., it provides an actual algorithm to find the unique global optimal solution over a  graph $G_d=\{V,E_d\}$ with edge-disjoint nonambiguous cycles. The pseudocode of such an algorithm is reported in Algorithm~\ref{algo:alg}.
Specifically, Algorithm~\ref{algo:alg}  consists of first finding all cycles and, for each cycle, selecting the link $(v_i,v_j)$ of smallest weight and then setting $x_{ji}=1$, paying a penalty  \textcolor{black}{equal to $\max\{\mathcal{A}_{\min},\delta\}$}.
After this operation, the graph conceptually becomes acyclic, and the algorithm sets to one the variables $x_{ij}$ for all remaining edges in $E_d$ \textcolor{black}{such that $\mathcal{A}_{ij}>1$.}
Finally, additional variables $x_{ij}=1$ are added by transitivity.
This is done by exploiting the properties of adjacency matrices. In particular, collecting the entries $x_{ij}$ into an adjacency matrix $\texttt{Adj}$, we have that $\texttt{Adj}^2_{ij}>0$ if and only if there is a path from node $v_i$ to node $v_j$ that passes through a third node $v_k$~\citep{godsil2001g}. Therefore, taking $sign(\texttt{Adj}+\texttt{Adj}^2)$ we introduce all variables $x_{ij}=1$ required to satisfy transitivity for the current variables $x_{ij}=1$. The procedure is repeated $n-1$ times to guarantee that also the newly added terms $x_{ij}=1$ satisfy transitivity.
Notice that, in general, the number of cycles $c$ in a directed graph can be exponential; however, if the cycles are edge-disjoint, it can be easily observed that there are at most $|E_d|/2$, i.e., $O(|E_d|)$ cycles.
At this point we observe that, using Johnson's algorithm, computing all cycles has a computational cost $O((|V|+|E_d|)(c+1))$ \citep{Johnson1975} and for each cycle we need to scan all edges, i.e., $O(|V|)$ operations in the worst case.
As for the addition of variables for transitivity, we observe that the matrix product has a computational complexity $O(|V|^{2.373})$ \citep{Davie} and that we compute such products $O(|V|)$ times; hence, the overall procedure has a computational complexity that is 
$$
\max\left\{O((|V|+|E_d|)(|E_d|+1)|V|),O(|V|^{3.373})\right\}\leq O(|V|^4)
$$
where the upper bound is obtained noting that $|E_d|\leq  n(n-1)/2$.
Therefore, under the assumptions of Theorem~\ref{theo:unique}, Algorithm~\ref{algo:alg} has polynomial time complexity.
\end{remark}
\begin{remark}
We reiterate that, given the fact $\sigma$ uses cardinal information to weight the ordinal preferences,  \textcolor{black}{and due to the presence of $\tau$}, Problem~\ref{prob:problem1} has a unique solution also in the presence of edge-disjoint nonambiguous cycles \textcolor{black}{and ties}. It can be noted that, by using an objective function based on purely ordinal information (e.g., the MV index), one can not guarantee unicity even under the assumptions of  Theorem~\ref{theo:unique}.
\end{remark}
\begin{remark}
The condition given in the above theorem is just a sufficient condition, hence the set of instances that correspond to a unique global optimal solution is larger. For instance, the example in Fig.~\eqref{fig:ambiguous2} consists of two ambiguous cycles sharing a link; yet, the global optimal solution is unique and corresponds to setting $x_{32}=1$ (paying a penalty $-\ln(3)$) and all other $x_{ij}$ corresponding to edges in $G_d$ to one. 
\end{remark}
\color{black}
\subsection{\textcolor{black}{ILLS} Ranking with Prescribed Pairwise Ordinal Preferences}
\textcolor{black}{Let us assume $\mathcal{A}$ satisfies the assumptions in Theorem~\ref{theo:unique} and let $\{x^*_{ij}\}$ be the optimal solution to the first subproblem.}
Within the second subproblem, our aim is to find a utility vector \mbox{${\bf w}^*=\exp({\bm y}^*)$}, where ${\bm y}^*$ solves the following problem.
\begin{problem}
Let $0 < \epsilon \ll 1$ be given. Find ${\bm y}^*\in\mathbb{R}^n$ that solves 
\label{prob:problem2}
\begin{equation}
\label{prob:convexproblem0}
\begin{aligned}
& \underset{{\bm y}\in\mathbb{R}^n}{\min\,} \sum_{i=1}^n \sum_{j\in \mathcal{N}_i}\left(\ln(\textcolor{black}{\mathcal{A}}_{ij})-y_i+y_j\right)^2&\\
& \text{subject to} &  \\
&\begin{cases}
y_i\geq y_j {\color{black} + \epsilon} , &\forall i,j,\,i\neq j \mbox{ s.t. } x^*_{ij}=1.
\end{cases}
\end{aligned}\end{equation}
\end{problem}
The above {\color{black}quadratic optimization} problem is essentially the classical logarithmic least-squares problem discussed in Section~\ref{subsec:log}, with the additional constraint able to guarantee that $w_i {\color{black}>} w_j$ whenever $x^*_{ij}=1$; the strict inequality in the constraint is implemented by introducing a small positive $\epsilon $.

Let us conclude the section by providing a necessary and sufficient global optimality condition for Problem~\ref{prob:problem2}.

{\color{black}

\begin{theorem}
Let us consider the AHP problem with incomplete information and let us assume that the graph $G$ corresponding to the ratio matrix ${\textcolor{black}{\mathcal{A}}}$ is connected and let $\{x^*_{ij}\}$ be the optimal solution to Problem~\ref{prob:problem1}. 
The global optimal solution ${\bm y}^*$ of Problem~\ref{prob:problem2} satisfies
\begin{equation}
\label{eq:zero1}
L(A){\bm y}^*=\dfrac{1}{2}(\Lambda^{*}-\Lambda^{*T}){\bm 1}_n+{\bm r}
\end{equation}
where $L(A)$ is the Laplacian matrix corresponding to the graph $G$ and $\Lambda^*$ is the $n\times n$ matrix of  Lagrange multipliers, such that for each pair of alternatives $i,j$ \textcolor{black}with $x^*_{ij}=1$ it holds
\begin{equation}
\label{eq:lambdaij}
\Lambda^*_{ij}=\max\left\{0,2\sum_{k\in \mathcal{N}_i,\,k\neq j}(y_i^*-y_k^*)-\sum_{k\neq j\,|\,x^*_{ik}=1}\Lambda^*_{ik} +\sum_{k\,|\,x^*_{ki}=1}\Lambda^*_{ki}-2r_i+2\epsilon\right\},
\end{equation}
while $\Lambda^*_{ij}=0$, otherwise,
with $r_i=\sum_{j\in \mathcal{N}_i}\ln(\textcolor{black}{\mathcal{A}}_{ij})$ 
and ${\bm r}=[r_1,\ldots,r_n]^T$.

\end{theorem}
\begin{proof}
Notice that, by construction, the problem is convex and has linear inequality constraints.
Moreover, since $\{x^*_{ij}\}$ is the optimal solution to Problem~\ref{prob:problem1}, by construction it is possible to assign values $y_i$ that satisfy the constraints in Problem~\ref{prob:problem2}; we conclude therefore, that the set of admissible solutions to Problem~\ref{prob:problem2} is nonempty.
In order to find the global optimal solution, we can therefore resort to the KKT first order criterion, which represents a necessary and sufficient global optimality condition\footnote{When the objective function is convex and the constraints are linear, in order to guarantee that the KKT first order criterion is a necessary and sufficient global optimality conditon, there is no need to check for constraint qualification conditions such as the Slater's condition, (see, for instance, \citep{Zangwill}); it is sufficient to show that the set of admissible solutions is nonempty.}.
The Lagrangian function associated to the problem at hand is:

$$
\mathcal{L}({\bm y},\Lambda)= \sum_{i=1}^n \sum_{k\in \mathcal{N}_i}\left(\ln(\textcolor{black}{\mathcal{A}}_{ik})-y_i+y_k\right)^2+
\sum_{i=1}^n \sum_{k\,|\,x_{ik}^*=1} \Lambda_{ik}(y_k-y_i+\epsilon)
$$
Following standard KKT theory, a necessary and sufficient optimality condition for ${\bm y}^*$ to be the global optimum is that there is $\Lambda^*$ such that
\begin{equation}
\label{eq:zero}
\begin{cases}\frac{\partial \mathcal{L}({\bm y},\Lambda)}{\partial y_i}\Big|_{{\bm y}={\bm y}^*,\Lambda=\Lambda^*}=
0,& \forall i\in\{1,\ldots,n\},\\[2em]
\Lambda^*_{ij}(y^*_j-y^*_i+\epsilon)=0,& \forall i,j  \mbox{ s.t. } x^*_{ij},\\[2em]
\Lambda^*_{ij}\geq0,& \forall i,j  \mbox{ s.t. } x^*_{ij}.
\end{cases}
\end{equation}
Note that, for all $i\in\{1,\ldots,n\}$ the fist of the above condition corresponds to 
\begin{equation}
\label{eq:first}
2\sum_{k\in \mathcal{N}_i}(y_i^*-y_k^*)=\sum_{k\,|\,x^*_{ik}=1}\Lambda^*_{ik} -\sum_{k\,|\,x^*_{ki}=1}\Lambda^*_{ki}+2\sum_{k\in \mathcal{N}_i}\ln(\textcolor{black}{\mathcal{A}}_{ik}).
\end{equation}
Let us now consider the second condition in Eq.~\eqref{eq:zero}; for all pairs of alternatives $i,j$ such that $x^*_{ij}=1$ it either holds $\Lambda^*_{ij}=0$ or
\begin{equation}
\label{eq:second}
y_j-y_i+\epsilon=0.
\end{equation}
In the latter case, since by combining Eq.~\eqref{eq:first} it holds
$$
2(y_j-y_i)=2\sum_{k\in \mathcal{N}_i,\,k\neq j}(y_i^*-y_k^*)-\sum_{k\,|\,x^*_{ik}=1}\Lambda^*_{ik} +\sum_{k\,|\,x^*_{ki}=1}\Lambda^*_{ki}-2\sum_{k\in \mathcal{N}_i}\ln(\textcolor{black}{\mathcal{A}}_{ik}),
$$
Eq.~\eqref{eq:second} is equivalent to requiring that
$$
2\sum_{k\in \mathcal{N}_i,\,k\neq j}(y_i^*-y_k^*)-\sum_{k\,|\,x^*_{ik}=1}\Lambda^*_{ik} +\sum_{k\,|\,x^*_{ki}=1}\Lambda^*_{ki}-2\sum_{k\in \mathcal{N}_i}\ln(\textcolor{black}{\mathcal{A}}_{ik})+2\epsilon=0,
$$
 i.e.,
 $$
\Lambda^*_{ij}=2\sum_{k\in \mathcal{N}_i,\,k\neq j}(y_i^*-y_k^*)-\sum_{k\neq j\,|\,x^*_{ik}=1}\Lambda^*_{ik} +\sum_{k\,|\,x^*_{ki}=1}\Lambda^*_{ki}-2\sum_{k\in \mathcal{N}_i}\ln(\textcolor{black}{\mathcal{A}}_{ik})+2\epsilon.
$$
We conclude that, since by the third condition in Eq.~\eqref{eq:zero} it must hold $\Lambda^*_{ij}\geq 0$, the Lagrange multiplier $\Lambda_{ij}^*$ satisfies Eq.~\eqref{eq:lambdaij} for all $i,j$ such that $x^*_{ij}=1$.
The proof is complete.
\end{proof}
}

\section{Experimental Results}

\subsection{Illustrative Examples}
In order to demonstrate the \textcolor{black}{ILLS-MWOV} methodology, \textcolor{black}{we now consider two illustrative examples.
Let us first} consider the example in \citep[Example 3.4]{csato2016incomplete}, for which the \textcolor{black}{ILLS} approach is known to yield an ordering ${\bf w}^{\textcolor{black}{ILLS}}$ that contradicts the ordinally consistent preferences $\{x_{ij}^{\textcolor{black}{ILLS}}\}$, in that ${\textcolor{black}{\mathcal{A}}}_{12}>\color{black}1\color{black}$ but $w^{\textcolor{black}{ILLS}}_1<w^{\textcolor{black}{ILLS}}_2$.
Specifically, \textcolor{black}{with reference to the results in \citep{csato2016incomplete}, the example encompasses $7$ alternatives and the graph underlying the available ratios is given in Figure~\ref{fig:res}.(a), concerning the weight matrix\footnote{ \textcolor{black}{The example in  \citep{csato2016incomplete} is given for generic coefficients $b,1/b$, in this case we set $b=2$.}}  ${\textcolor{black}{\mathcal{A}}}$, it is defined as:}

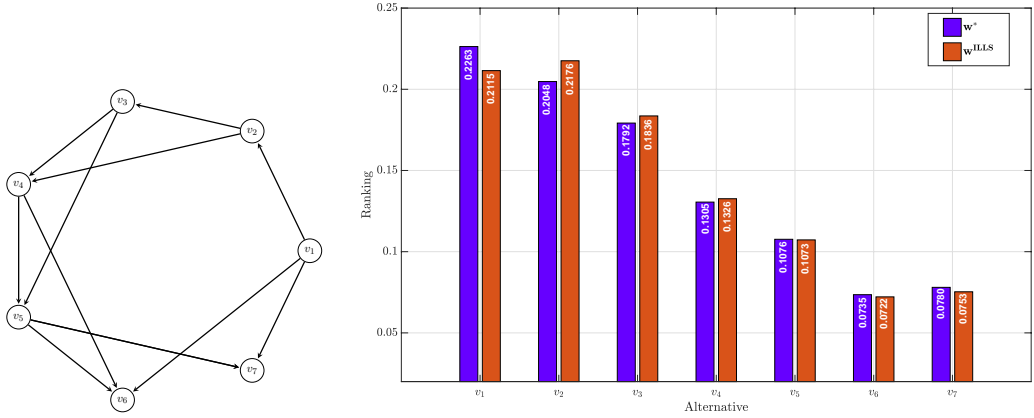
\begin{figure*}
\centering
\subfigure[\textcolor{black}{Graph representation of the instance considered in \citep{csato2016incomplete}.}]{%
\resizebox {.3\textwidth} {!}{
\begin{tikzpicture}[
            > = stealth, 
            shorten > = 1pt, 
            auto,
            node distance = 1cm, 
            semithick 
        ]

        \tikzstyle{every state}=[
            draw = black,
            thick,
            fill = white,
            minimum size = 4mm,
            minimum width=5mm
        ]

 	\node[state] (v1) at ({0*51.4285714}:5){$v_1$};
        \node[state] (v2) at ({1*51.4285714}:5){$v_2$};
        \node[state] (v3)at ({2*51.4285714}:5){$v_3$};
        \node[state] (v4) at ({3*51.4285714}:5) {$v_4$};
        \node[state] (v5) at ({4*51.4285714}:5) {$v_5$};
        \node[state] (v6) at ({5*51.4285714}:5) {$v_6$};
  	\node[state] (v7) at ({6*51.4285714}:5) {$v_7$};
             
\path[sloped,->, above,very thick] (v1) edge node {\begin{footnotesize}\end{footnotesize}} (v2);
\path[sloped,->, above,very thick] (v2) edge node {\begin{footnotesize}\end{footnotesize}} (v3);
\path[sloped,->, above,very thick] (v3) edge node {\begin{footnotesize}\end{footnotesize}} (v4);
\path[sloped,->, above,very thick] (v4) edge node {\begin{footnotesize}\end{footnotesize}} (v5);
\path[sloped,->, above,very thick] (v5) edge node {\begin{footnotesize}\end{footnotesize}} (v6);
\path[sloped,->, above,very thick] (v5) edge node {\begin{footnotesize}\end{footnotesize}} (v7);
\path[sloped,->, above,very thick] (v1) edge node {\begin{footnotesize}\end{footnotesize}} (v7);
\path[sloped,->, above,very thick] (v2) edge node {\begin{footnotesize}\end{footnotesize}} (v4);
\path[sloped,->, above,very thick] (v3) edge node {\begin{footnotesize}\end{footnotesize}} (v5);
\path[sloped,->, above,very thick] (v4) edge node {\begin{footnotesize}\end{footnotesize}} (v6);
\path[sloped,->, above,very thick] (v5) edge node {\begin{footnotesize}\end{footnotesize}} (v7);
\path[sloped,->, above,very thick] (v1) edge node {\begin{footnotesize}\end{footnotesize}} (v6);
\end{tikzpicture}
}
}\hspace{5pt}
\subfigure[Rankings obtained via the \textcolor{black}{Incomplete} Logarithmic Least-squares (${\bf w}^{\textcolor{black}{ILLS}}$) and the proposed approach for the ordinal ranking ${\bf w}^{*}$.]{%
\resizebox*{9cm}{!}{\includegraphics{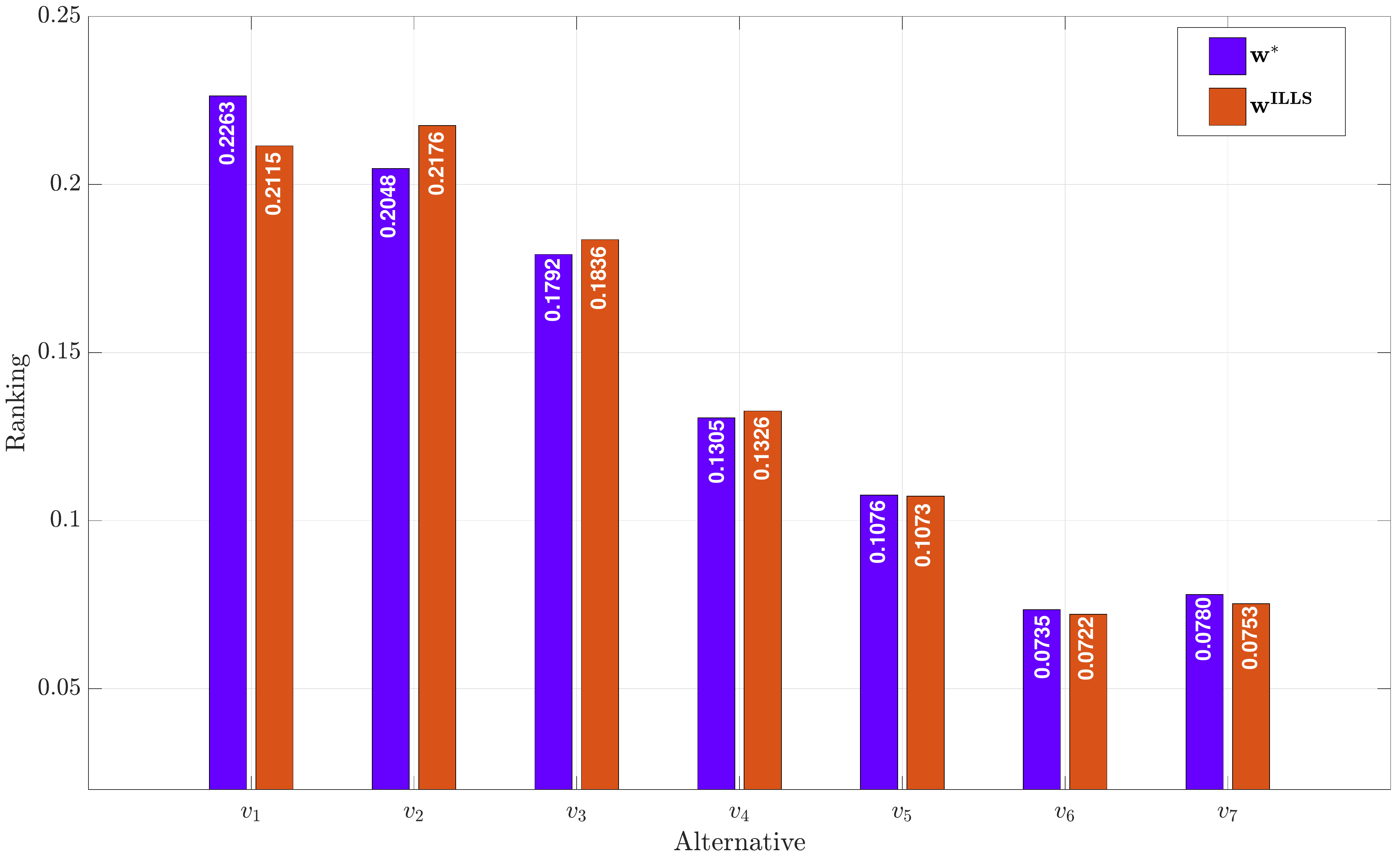}}}
\caption{\textcolor{black}{Comparison of the results of ILLS-MWOV and ILLS with respect to the example given in \citep{csato2016incomplete}, for which ILLS is known to result in ordinal violations.}}
\label{fig:res}
\end{figure*}

$$
\textcolor{black}{\mathcal{A}}=\begin{bmatrix}
\color{black}1\color{black}&2&0&0&0&2&2\\
    {1}/{2}&\color{black}1\color{black}&2&2&0& 0 &0\\
    0 &{1}/{2}& \color{black}1\color{black}& 2 &2& 0 &0\\
    0 &{1}/{2}& {1}/{2} &\color{black}1\color{black}& 2& 2& 0\\
    0 &0 &{1}/{2}&{1}/{2}& \color{black}1\color{black}& 2& 2\\
    {1}/{2}&0& 0 &{1}/{2} &{1}/{2}& \color{black}1\color{black}& 0\\
    {1}/{2}&0 &0& 0& {1}/{2}& 0 &\color{black}1\color{black}
\end{bmatrix}.
$$
Figure~\ref{fig:res}.(b) shows the ranking ${\bf w}^{\textcolor{black}{ILLS}}$  and ${\bf w}^*$ obtained via the \textcolor{black}{incomplete} logarithmic least squares approach and via the \textcolor{black}{ILLS-MWOV} approach {\color{black}(considering $\epsilon$ = 0.1)},  with blue and red bars, respectively.
\color{black}Notice that no entry $\mathcal{A}_{ij}$ with $i\neq j$ is equal to one; hence, we have $\tau=0$.
Notice further that the ranking ${\bf w}^{\textcolor{black}{ILLS}}$ results in one violation of the ordering, since $\textcolor{black}{\mathcal{A}}_{12}>\color{black}1\color{black}$ but $w^{\textcolor{black}{ILLS}}_1<w^{\textcolor{black}{ILLS}}_2$; in other words, it holds $\sigma^{\textcolor{black}{ILLS}}=10\ln(2)\approx6.931$, while the objective function of Problem~\ref{prob:problem2} is equal to $1.6963$.
\color{black}
Let us now consider the result of the \textcolor{black}{ILLS-MWOV}. Specifically, by solving Problem~\ref{prob:problem11}, we obtain a pairwise ordering $\{x^*_{ij}\}$ that can be summarized by the $n\times n$ matrix $X^*$ such that $X^*_{ij}=x^*_{ij}$, i.e., 
$${\color{black}
X^* =\begin{bmatrix}
     0   &  1   &  1   &  1  &   1 &   1   &  1\\
     0    & 0    & 1    & 1    & 1    &  1    & 1\\
     0    & 0  &   0   &  1   &  1    & 1   & 1\\
     0    & 0   &  0   &  0  &   1   &  1   &  1\\
     0    & 0    & 0   &  0  &   0  &   1  &   1\\
     0   &  0 &    0  &   0   &  0  &   0  &   0\\
     0    & 0   &  0  &   0 &    0  &   0   &  0
\end{bmatrix}.}
$$
Notice that $X^*$ is in accordance with the pairwise ordinal preferences  induced by ${\textcolor{black}{\mathcal{A}}}$. Moreover, it holds \color{black}$\sigma^*=11\ln(2)\approx7.6246$.
Notice that, since the graph is acyclic, we are guaranteed by Theorem~\ref{theo:unique} that the solution found with the proposed approach is unique. \color{black}
Let us now consider the solution of Problem~\ref{prob:problem2}, where the constraints depend on the above choice for $\{x^*_{ij}\}$; the resulting ranking is shown in Figure~\ref{fig:res}.(b).
{\color{black}Notice that, in contrast to the relation $w_1^{\textcolor{black}{ILLS}} < w_2^{\textcolor{black}{ILLS}}$, the result of the proposed \textcolor{black}{ILLS-MWOV} approach is characterized by $w^*_1>w^*_2$; hence, it preserves the relations between the two alternatives.}
Notably, the objective function of Problem~\ref{prob:problem2} is equal to {\color{black}$1.7232$, i.e., an increase of just $+1.6\%$ with respect to the results obtained for ${\bf w}^{\textcolor{black}{ILLS}}$. We can affirm that the distribution of the weights {$\bf w^*$}, with respect to the distribution ${\bf w}^{\textcolor{black}{ILLS}}$, is slightly suboptimal but also preserves the ordinal constraints. 

For the sake of completeness, we adopt the criteria described in Section~\ref{sec:criteria} in order to evaluate the performance of the proposed \textcolor{black}{ILLS-MWOV} approach with respect to the classic \textcolor{black}{ILLS} method.
Concerning the MVs criterion (see Section~\ref{sec:MV}), it confirms the results obtained by analyzing the \textcolor{black}{weighted} ordinal satisfaction index $\sigma$; in fact, the criterion computed for both the approaches yields $MVs^{\textcolor{black}{ILLS}} = 1$ (due to the incorrect order between the first two alternatives) and $MVs^{\textcolor{black}{ILLS-MWOV}} = 0$ (because of the correct ordering between the first two alternatives). 

Let us  compare the two approaches in terms of the TDs criterion (see Section~\ref{sec:TD}). In this case we obtain $TDs^{\textcolor{black}{ILLS}} = 5.71$, while $TDs^{\textcolor{black}{ILLS-MWOV}} = 6.17$. This result {\color{black} is expected due to the additional constraints introduced in the proposed formulation}. Let us now analyze the element-wise absolute differences $\Delta_{ij}$ for the available entries, which we define as 
$$
\Delta_{ij}=\begin{cases}
\left |{ \frac{ w^{\textcolor{black}{ILLS}}_i}{ w^{\textcolor{black}{ILLS}}_j}} - { \frac{ w^{*}_i}{ w^{*}_j}}\right |,& \mbox{ if } \textcolor{black}{\mathcal{A}}_{ij}\neq 0\\
0,& \mbox{ otherwise.}
\end{cases}
$$ 
Collecting all entries $\Delta_{ij}$ into the $n\times n$ matrix $\Delta$ we have that
$$
\Delta  =
\begin{bmatrix} 
0&	0.133&	0&	0&	0&	0.150&	0.093\\
0.124&	0&	0.042&	0.072&	0&	0&	0\\
0&	0.031&	0&	0.011&	0.046&	0&	0\\
0&0.028&	0.006&0&	0.023&	0.061&	0\\
0&	0&	0.016&	0.015&	0&	0.022&	0.045\\
0.017&	0&	0&	0.019&	0.010&0&	0\\
0.011&	0&	0&	0&     0.023&	0&	0\\
 \end{bmatrix}
$$
and we observe that the largest  absolute differences are attained for entries that involve $v_1$ or $v_2$ (in particular, the largest attained values are $\Delta_{16}=0.15,$ $\Delta_{12}=0.133$ and $\Delta_{21}=0.124$), i.e., the pair for which the ordinal preferences encoded in $\textcolor{black}{\mathcal{A}}$ are violated using the \textcolor{black}{ILLS} approach. 

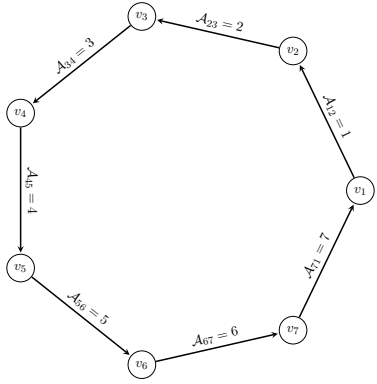
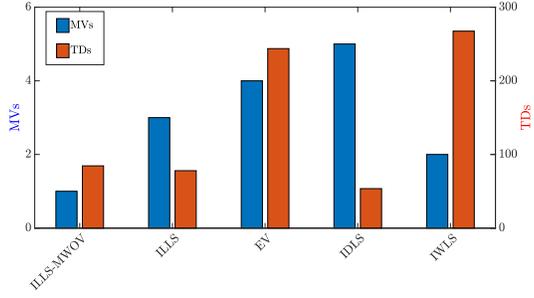
\begin{figure*}
\centering
\subfigure[Example of nonambiguous cycle]{%
\resizebox {.35\textwidth} {!}{
\begin{tikzpicture}[
            > = stealth, 
            shorten > = 1pt, 
            auto,
            node distance = 0.7cm, 
            semithick 
        ]

        \tikzstyle{every state}=[
            draw = black,
            thick,
            fill = white,
            minimum size = 4mm,
            minimum width=5mm
        ]

 	\node[state] (v1) at ({0*51.4285714}:5){$v_1$};
        \node[state] (v2) at ({1*51.4285714}:5){$v_2$};
        \node[state] (v3)at ({2*51.4285714}:5){$v_3$};
        \node[state] (v4) at ({3*51.4285714}:5) {$v_4$};
        \node[state] (v5) at ({4*51.4285714}:5) {$v_5$};
        \node[state] (v6) at ({5*51.4285714}:5) {$v_6$};
  	\node[state] (v7) at ({6*51.4285714}:5) {$v_7$};
             
\path[sloped,->, above,very thick] (v1) edge node {$\mathcal{A}_{12}=1$} (v2);
\path[sloped,->, above,very thick] (v2) edge node {$\mathcal{A}_{23}=2$} (v3);
\path[sloped,->, above,very thick] (v3) edge node {$\mathcal{A}_{34}=3$} (v4);
\path[sloped,->, above,very thick] (v4) edge node {$\mathcal{A}_{45}=4$} (v5);
\path[sloped,->, above,very thick] (v5) edge node {$\mathcal{A}_{56}=5$} (v6);
\path[sloped,->, above,very thick] (v6) edge node {$\mathcal{A}_{67}=6$} (v7);
\path[sloped,->, above,very thick] (v7) edge node {$\mathcal{A}_{71}=7$} (v1);
\end{tikzpicture}
}
}\hspace{5pt}
\subfigure[Results in terms of MVs and TDs obtained via the proposed approach and using other methods in the literature.]{%
\resizebox*{7cm}{!}{\includegraphics{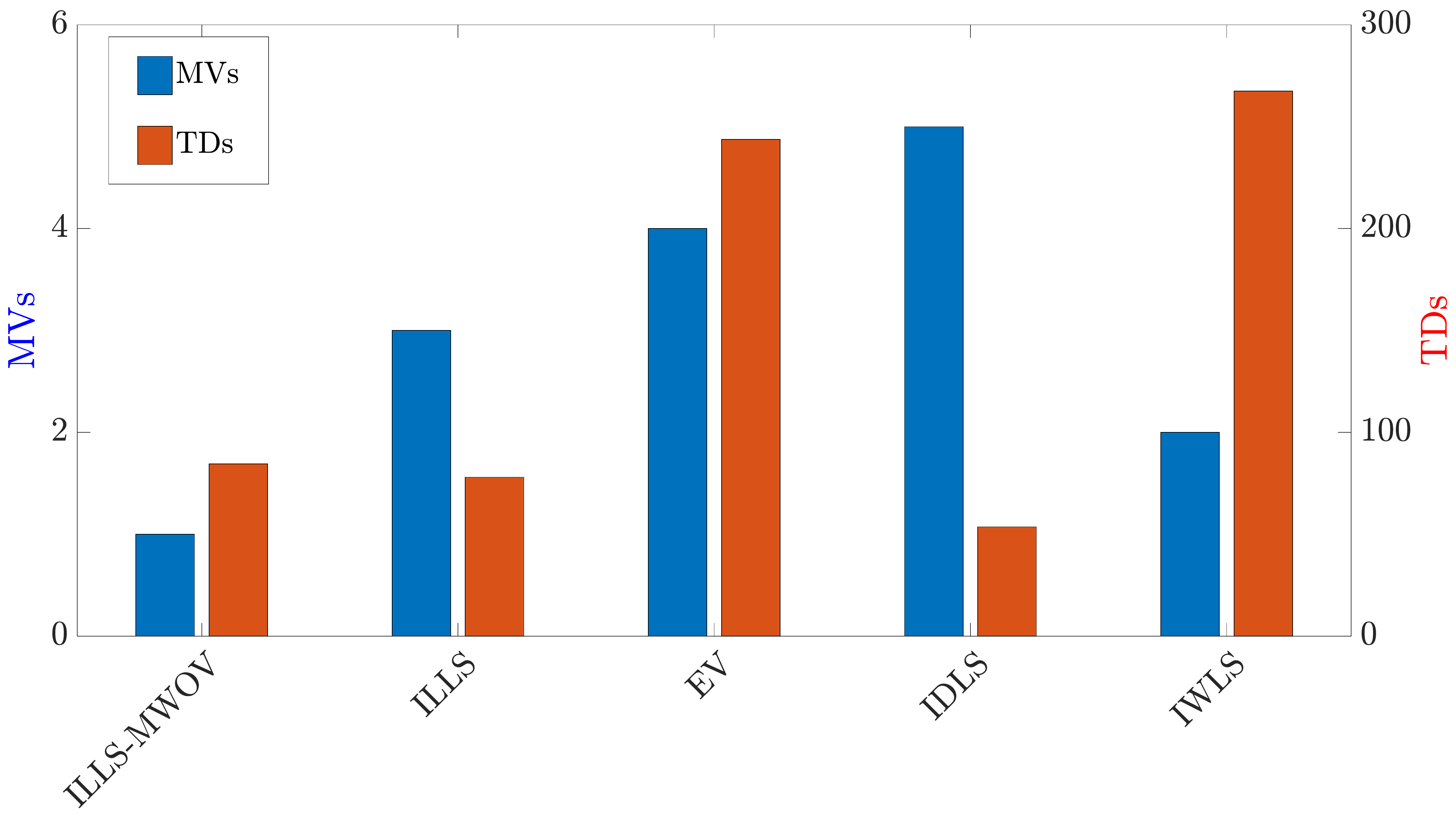}}}
\caption{\textcolor{black}{Results in terms of MVs and TDs for different methodologies with respect to a nonambiguous cycle featuring a tie.}}
\label{fig:ciclononambiguo}
\end{figure*}

\textcolor{black}{Let us now consider the graph reported in Figure~\ref{fig:ciclononambiguo}.(a), i.e., a non-ambiguous cycle featuring a tie.
With respect to this instance, the ILLS-MWOV solution is such that $x_{21}=1$ (paying a penalty equal to $\delta\ll 1$)  and all other $x_{i,i+1}=1$ (plus the additional terms $x_{ij}$ required to enforce transitivity).
In Figure~\ref{fig:ciclononambiguo}.(b) we show the results in terms of MVs and TDs obtained according to ILLS-MWOV, ILLS, EV, IDLS and IWLS.
According to the figure, the proposed approach results in remarkably smaller values of MVs, which is related to the satisfaction of ordinal preferences, while the results in terms of TDs (related to cardinal preferences) is slightly larger than  ILLS and IDLS, but remarkably smaller than EV and IWLS.
Overall, this suggests that the proposed approach has remarkably better performance with respect to metrics related to ordinal information and good performances with respect to metrics related to cardinal information.}

{\color{black}
\subsection{Comparison with the State of the Art}
\label{sec:compar}
In order to experimentally validate the proposed approach, we compare in this section the \textcolor{black}{ILLS-MWOV} methodology with \textcolor{black}{ILLS} (see Section~\ref{subsec:log}), \textcolor{black}{IDLS} (see Section~\ref{sec:sdls}), \textcolor{black}{IWLS} (see Section~\ref{sec:swls}) and EV (see Section~\ref{sec:ev}) in terms of metrics related to the ordinal violations (i.e.,  $\sigma$ and MVs) and in terms of cardinal violations (i.e., TDs). Such a comparison is undertaken by considering the effect of growing levels of inconsistency and varying the density $\rho$ of the graph (\textcolor{black}{we reiterate that the graph density is the ratio between the cardinality of the edge set and the cardinality of the edges in a compete graph}). In more detail, we consider random instances encompassing $n=7$ alternatives; for each instance we build a connected random graph with the desired graph density $\rho$ (i.e., the percentage of links with respect to those in a complete graph with the same number of nodes) and we generate a nominal weight vector ${\bf w}$, which we use to  construct a nominal matrix $\textcolor{black}{\mathcal{A}}$ such that $\textcolor{black}{\mathcal{A}}_{ij}=w_i/w_j$ whenever $(v_i,v_j)\in E$. Then, we perturbate each nominal ratio $\textcolor{black}{\mathcal{A}}_{ij}$ by setting $$\widetilde{\textcolor{black}{\mathcal{A}}}_{ij}=\textcolor{black}{\mathcal{A}}_{ij}e^{\eta_{ij}},$$ $\eta_{ij}$ being a normal random variable with zero mean and standard deviation $\gamma$, which we refer to as the \textcolor{black}{\em degree of perturbation}}. In other words, $e^{\eta_{ij}}$ is a log-normal perturbation which has the effect to attenuate the ratio  $\textcolor{black}{\mathcal{A}}_{ij}$ (i.e., when $\eta_{ij}$ assumes negative values) or to enhance the ratio  $\textcolor{black}{\mathcal{A}}_{ij}$ (i.e., when $\eta_{ij}$ assumes positive values).
Notice that, in order to ensure local consistency (i.e., consistency at the level of each single pair of alternatives), we only directly perturbate $\widetilde{\textcolor{black}{\mathcal{A}}}_{ij}$, while we set $\widetilde{\textcolor{black}{\mathcal{A}}}_{ji}=1/\widetilde{\textcolor{black}{\mathcal{A}}}_{ij}$. In the trials, we set the parameters $\epsilon,\delta$ in \textcolor{black}{ILLS-MWOV} to $\epsilon=\delta=10^{-4}$, and we show the results in terms of average and standard deviation over $m=100$ trials for each choice of \textcolor{black}{the parameters $\rho$ and $\gamma$}. 

\begin{figure*}
\begin{center} 
\includegraphics[width=.99\textwidth]{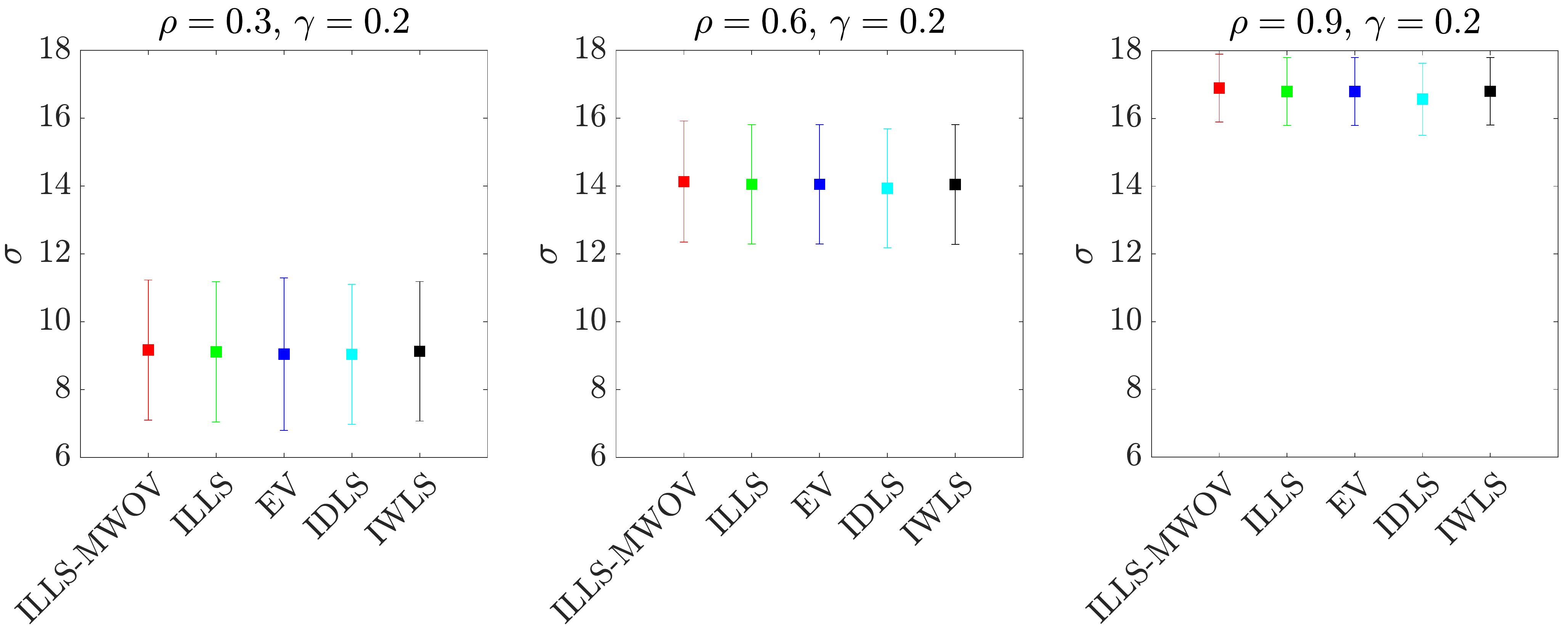}
\caption{\textcolor{black}{Weighted} Ordinal Satisfaction Index $\sigma$ comparison among \textcolor{black}{ILLS-MWOV, ILLS, EV, IDLS and IWLS} for multiple values of graph density $\rho$. For each choice of density we show the results in terms of average and standard deviation over $m=100$ trials, setting $\epsilon=\delta=10^{-4}$ and $\gamma=0.2$. \textcolor{black}{The parameters $\gamma,\rho$ adopted in each simulation are reported above the corresponding plot.}}
 \label{fig:XXX}
\end{center}
\end{figure*}

\begin{figure*}
\begin{center} 
\includegraphics[width=.99\textwidth]{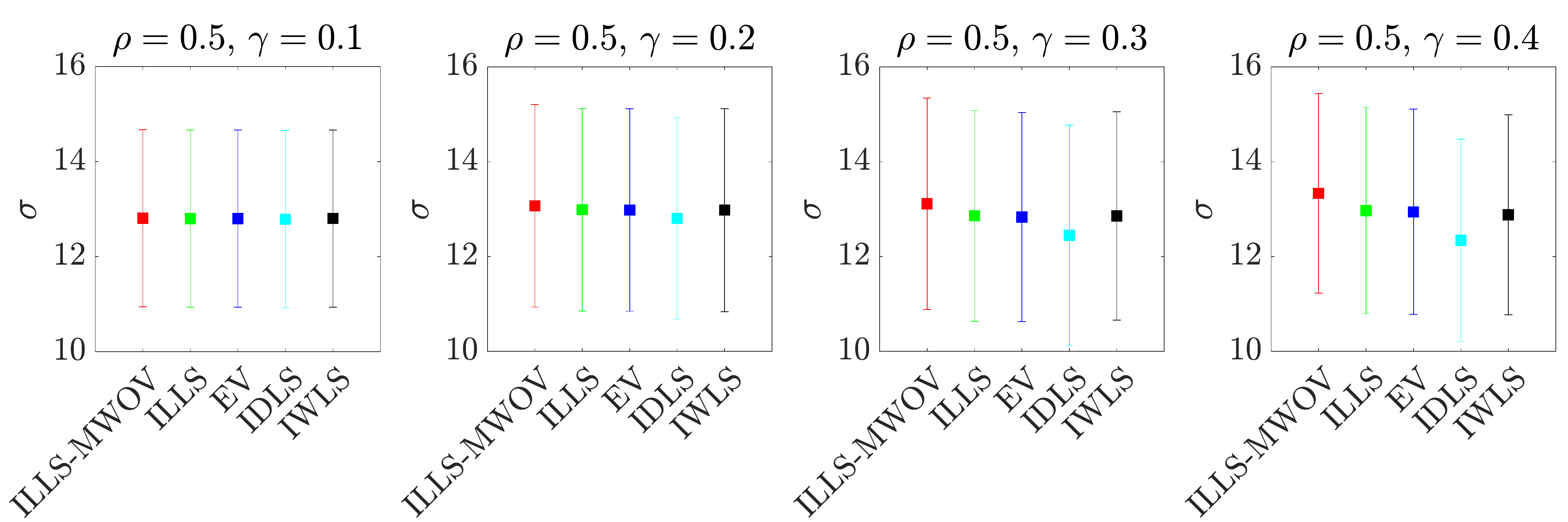}
\caption{\textcolor{black}{Weighted} Ordinal Satisfaction Index $\sigma$ comparison among \textcolor{black}{ILLS-MWOV}, \textcolor{black}{ILLS, EV, IDLS and IWLS} for multiple values of the \textcolor{black}{degree of perturbation} $\gamma$. For each choice of $\gamma$ we show the results in terms of average and standard deviation over $m=100$ trials, setting $\epsilon=\delta=10^{-4}$ and the graph density to $\rho=0.5$. \textcolor{black}{The parameters $\gamma,\rho$ adopted in each simulation are reported above the corresponding plot.}}
 \label{fig:YYY}
\end{center}
\end{figure*}

Let us now consider the performances in terms of the index\footnote{In the simulations we always have $\mathcal{A}_{ij}\neq 1$ for $i\neq j$; hence, $\tau=0$.}  $\sigma$\ (we recall that the larger is $\sigma$, the more the result matches with the ordinal preferences encoded by the ratios $\textcolor{black}{\mathcal{A}}_{ij}$), considering graphs with different \textcolor{black}{values of $\rho$ with fixed $\gamma$} (Figure~\ref{fig:XXX}) and \textcolor{black}{for different values of $\gamma$ with fixed $\rho$} (Figure~\ref{fig:YYY}). 
\textcolor{black}{According to the plots, while performances are comparable with the other methods in the case of small $\gamma$, irrespectively of $\rho$ (i.e., Figure~\ref{fig:XXX} and the leftmost plot in Figure~\ref{fig:YYY}), as $\gamma$ grows (see the central and right panels in Figure~\ref{fig:YYY}), the divide with other methods appears more evident, especially with respect to IDLS.
Notice that, according to Figure~\ref{fig:XXX}, we observe that when $\gamma$ is constant and $\rho$ grows, the index $\sigma$ tends to grow for all considered methods (this is due to the increased number of links, which results in an increased number of pairs for which a reward is obtained within $\sigma$).
As for the effect of growing $\gamma$ with fixed $\rho$, we observe in Figure~\ref{fig:YYY} that the proposed method outperforms all others, especially when $\gamma$ is large. In fact, the increased amount of links is likely to introduce further violations of the ordinal preferences, thus calling for a methodology able to deal with ordinal information.}

\begin{figure*}
\begin{center} 
\includegraphics[width=.99\textwidth]{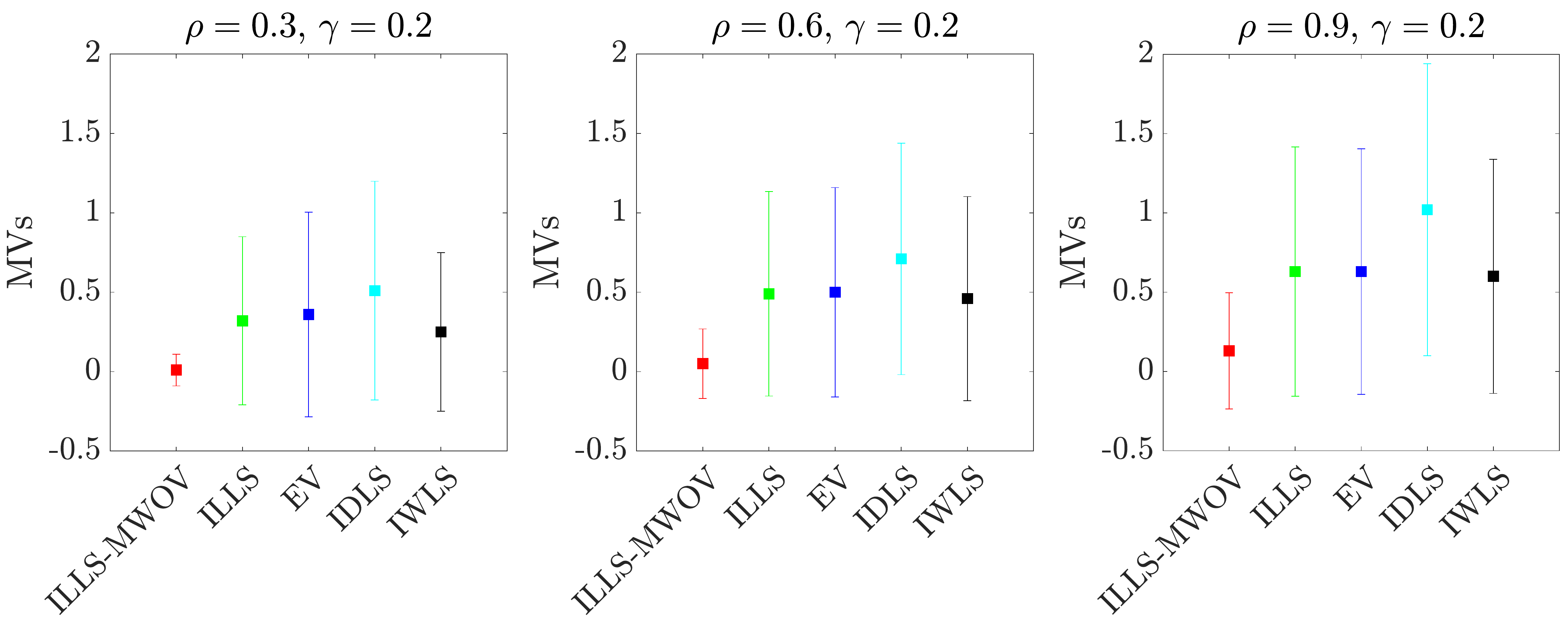}
\caption{Maximum Violations (MVs)  comparison among \textcolor{black}{ILLS-MWOV}, \textcolor{black}{ILLS, EV, IDLS and IWLS} for multiple values of graph density $\rho$. For each choice of density we show the results in terms of average and standard deviation over $m=100$ trials, setting $\epsilon=\delta=10^{-4}$ and $\gamma=0.2$. \textcolor{black}{The parameters $\gamma,\rho$ adopted in each simulation are reported above the corresponding plot.}}
 \label{fig:MV_dens}
\end{center}
\end{figure*}

\begin{figure*}
\begin{center} 
\includegraphics[width=.99\textwidth]{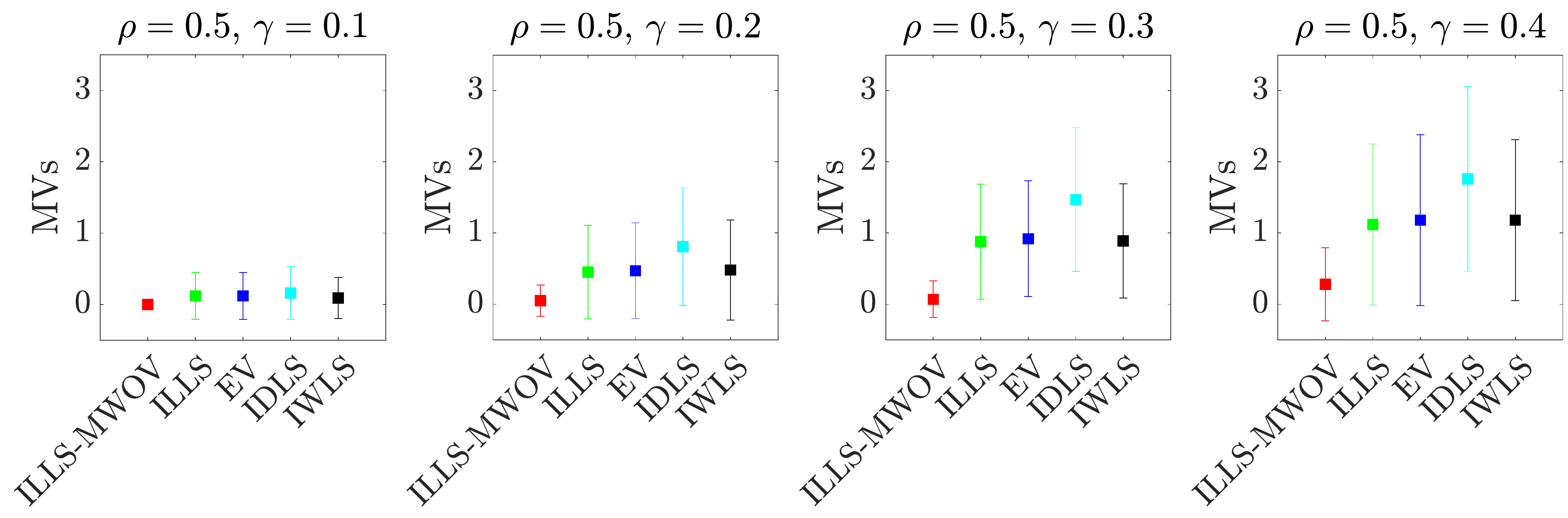}
\caption{Maximum Violations (MVs) comparison among \textcolor{black}{ILLS-MWOV}, \textcolor{black}{ILLS, EV, IDLS and IWLS} for multiple values of the \textcolor{black}{degree of perturbation} $\gamma$. For each choice of the inconsistency we show the results in terms of average and standard deviation over $m=100$ trials, setting $\epsilon=\delta=10^{-4}$ and the graph density to  $\rho=0.5$. \textcolor{black}{The parameters $\gamma,\rho$ adopted in each simulation are reported above the corresponding plot.}}
 \label{fig:MV_inc}
\end{center}
\end{figure*}

In order to validate the ability of the proposed method to provide a ranking that is in agreement with the ordinal preferences, we now focus our attention on the MVs metric (in this case, the smaller is MVs the more the ranking is in accordance with the available ordinal information), considering the effect of increasing \textcolor{black}{$\rho$ with fixed $\gamma$} (Figure~\ref{fig:MV_dens}) and the effect of increasing \textcolor{black}{$\gamma$ with fixed $\rho$} (Figure~\ref{fig:MV_inc}).

We point out that,  differently from all other methods, \textcolor{black}{IDLS} has been specifically designed to optimize TDs.
Notice that none of the considered methods, including the proposed one, has been specifically designed to optimize MVs. 
According to the figures, in all cases the proposed \textcolor{black}{ILLS-MWOV} approach shows the best performance; the divide is particularly evident as \textcolor{black}{$\rho$ and $\gamma$ grow}.
Notice that the performance of all other methods is comparable, except for the \textcolor{black}{IDLS} approach, which yields the worst results.

\begin{figure*}
\begin{center} 
\includegraphics[width=.99\textwidth]{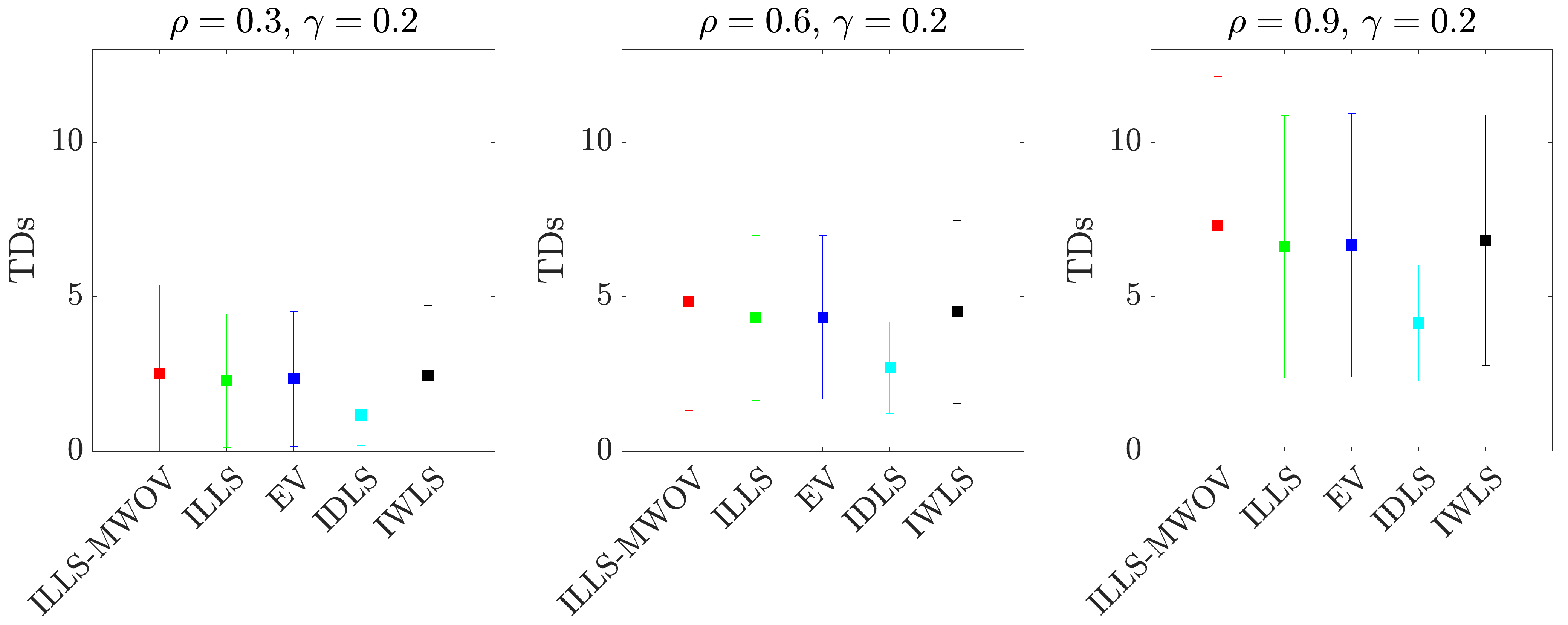}
\caption{Total Deviation (TDs) comparison among \textcolor{black}{ILLS-MWOV},  \textcolor{black}{ILLS, EV, IDLS and IWLS} for multiple values of graph density $\rho$. For each choice of density we show the results in terms of average and standard deviation over $m=100$ trials, setting $\epsilon=\delta=10^{-4}$ and $\gamma=0.2$. \textcolor{black}{The parameters $\gamma,\rho$ adopted in each simulation are reported above the corresponding plot.}}
 \label{fig:TD_dens}
\end{center}
\end{figure*}

\begin{figure*}
\begin{center} 
\includegraphics[width=.99\textwidth]{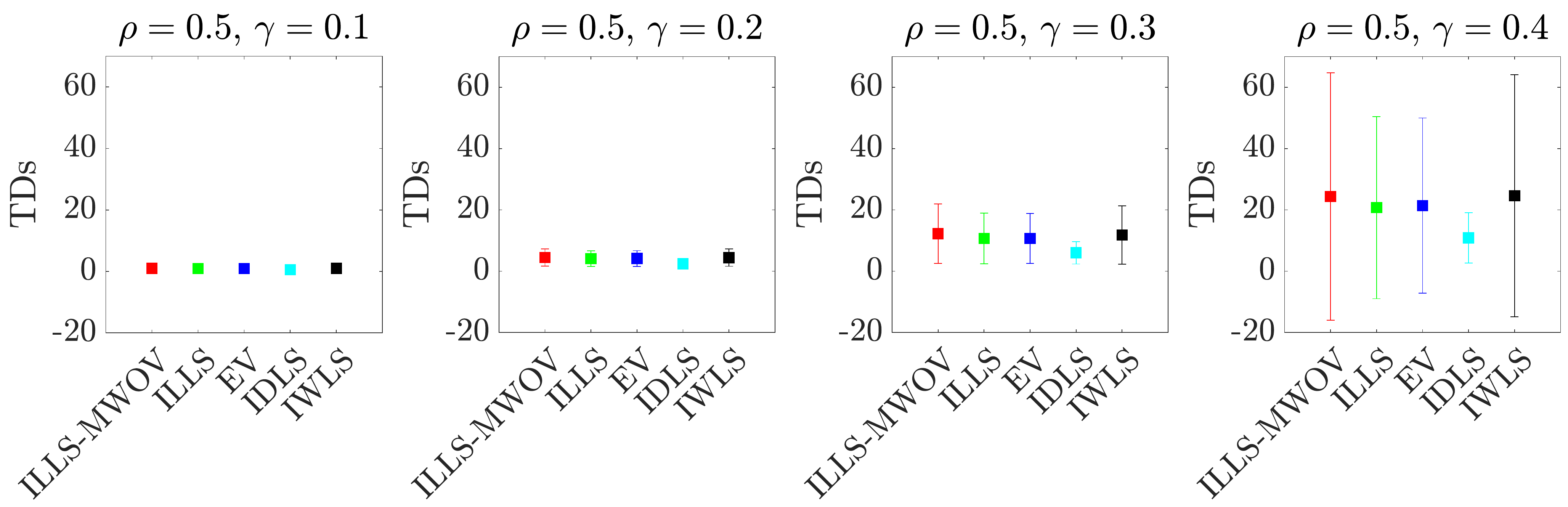}
\caption{Total Deviations (TDs) comparison among \textcolor{black}{ILLS-MWOV},  \textcolor{black}{ILLS, EV, IDLS and IWLS} for multiple values of \textcolor{black}{the degree of perturbation} $\gamma$. For each choice of the inconsistency we show the results in terms of average and standard deviation over $m=100$ trials, setting $\epsilon=\delta=10^{-4}$ and the graph density to  $\rho=0.5$. \textcolor{black}{The parameters $\gamma,\rho$ adopted in each simulation are reported above the corresponding plot.}}
 \label{fig:TD_inc}
\end{center}
\end{figure*}

Let us now consider the results obtained by the different methodologies in terms of satisfaction of the cardinal information encoded in the ratios $\textcolor{black}{\mathcal{A}}_{ij}$; to this end we compare the methods in terms of the TDs metric,
considering the effect of increasing \textcolor{black}{$\rho$ with fixed $\gamma$} (Figure~\ref{fig:TD_dens}) and the effect of increasing \textcolor{black}{$\gamma$ with fixed $\rho$} (Figure~\ref{fig:TD_inc}).
It can be noted that, although the proposed \textcolor{black}{ILLS-MWOV} approach is focused on the ordinal dimension, the results obtained are in all cases comparable to,  but slightly  worse than \textcolor{black}{ILLS}, EV and \textcolor{black}{IWLS}.

Overall, these simulations suggest that \textcolor{black}{ILLS-MWOV}  is very effective in minimizing violation of ordinal preferences; moreover, it is comparable to most of the approaches in the literature in terms of violation of cardinal preferences, while being outperformed only by an approach that explicitly focuses on cardinal information.
Figure~\ref{fig:pareto} summarizes such a situation by providing a synoptic view of the performance of the different algorithms in terms of ordinal and cardinal violations.
Specifically, \textcolor{black}{ILLS}, EV and \textcolor{black}{IWLS} have substantially comparable performance, and represent a tradeoff between ordinal and cardinal effectiveness. On the other and, as mentioned before, \textcolor{black}{IDLS} is focused on cardinal information and consequently it significantly suboptimal in terms of ordinal information.
Finally, the proposed approach exhibits an ``opposite" behavior with respect to \textcolor{black}{IDLS}, emphasizing the agreement of the ranking with ordinal information.

%
%
%
%
%

}

\begin{figure*}
\begin{center} 
\includegraphics[width=.9\textwidth]{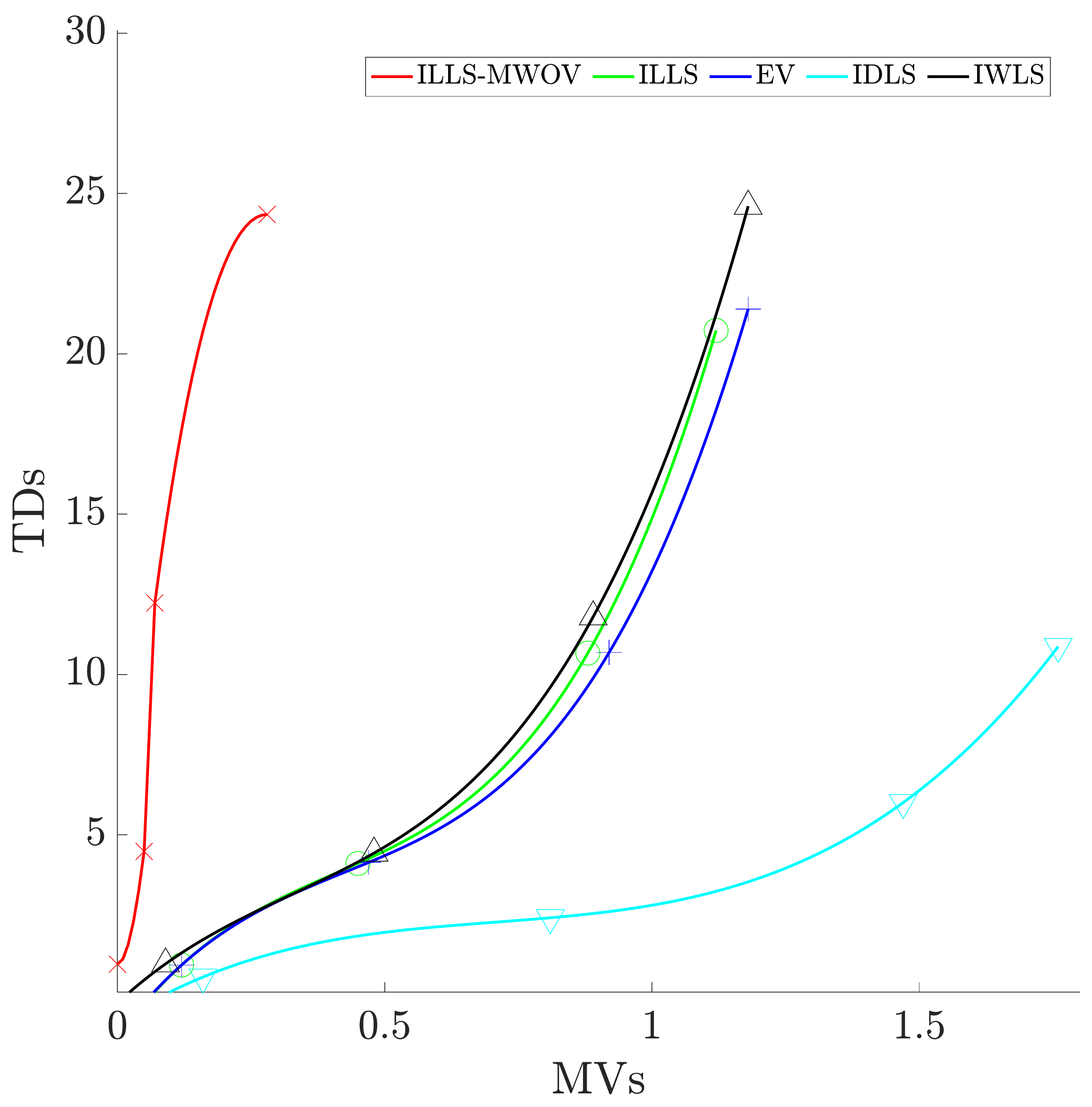}
\caption{At-a-glance view of the performance of the considered methodologies in terms of both MVs and TDs. For each method we consider $m=100$ graphs with fixed density equal to $0.5$ and  we plot the results in terms of the average MVs and TDs scores obtained for different levels of inconsistency using different markers (the continuous curves represent an interpolation of such values).}
 \label{fig:pareto}
\end{center}
\end{figure*}

\begin{figure*}
\begin{center} 
\includegraphics[width=.9\textwidth]{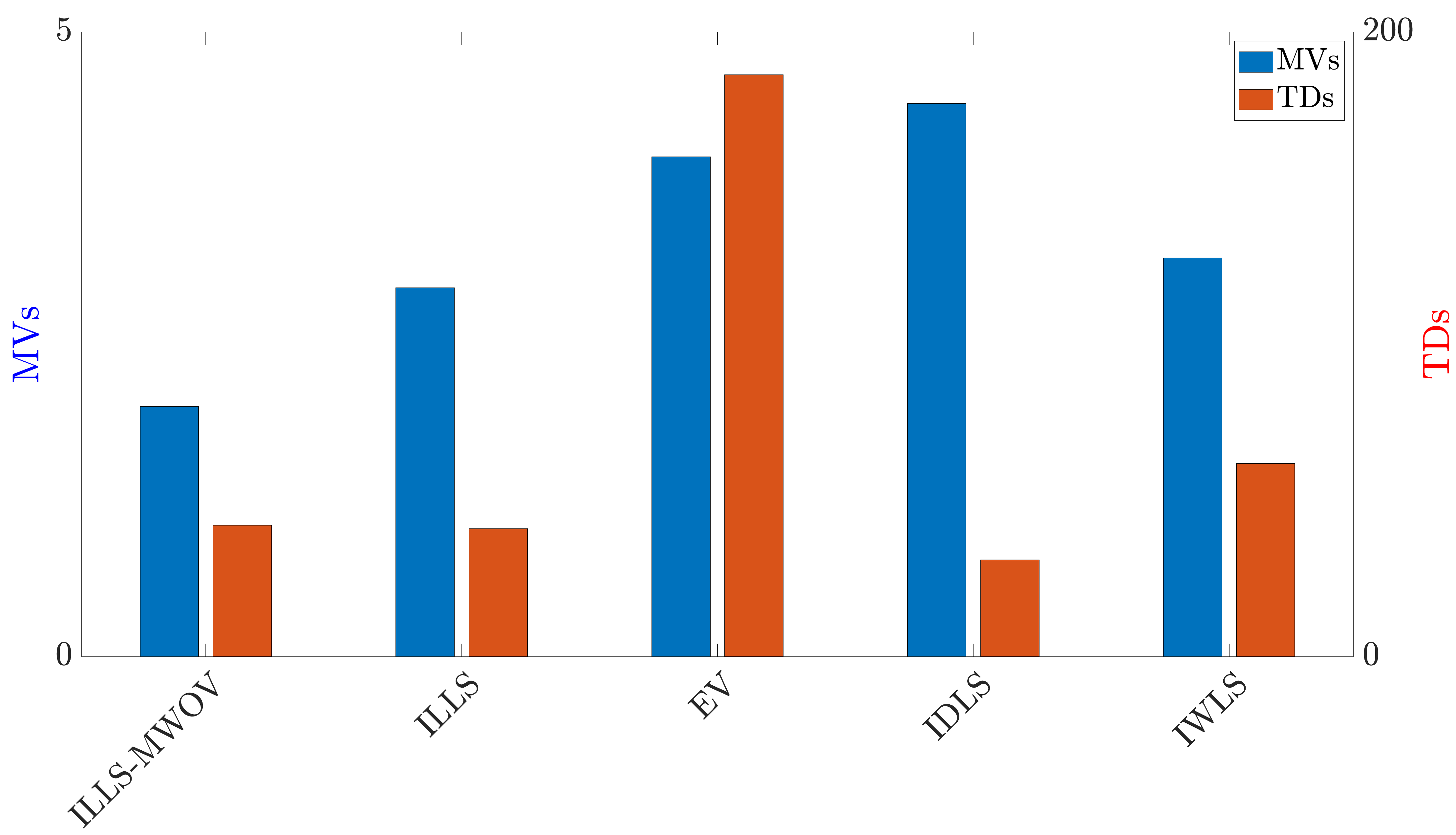}
\caption{\textcolor{black}{Comparison of ILLS-MWOV, ILLS, EV, IDLS and IWLS in terms of MVs and TDs with respect to $20$ instances obtained by permutation of the links in an ambiguous cycle.}}
 \label{fig:cicliambigui}
\end{center}
\end{figure*}

\color{black}{To conclude, we show in Figure~\ref{fig:cicliambigui} the results in terms of MVs and TDs of the different methodologies with respect to instances composed by a sigle ambiguous cycle with $n=7$ nodes and two links with associated weight $\mathcal{A}_{\min}=2$, while the other links have unique values in $\{3,4,\ldots,6\}$. In more detail, we report the average values of MVs and TDs over $21$ graphs, which correspond to all possible permutations of the two weights equal to two in the cycle, while the other weights are assigned to the remaining links by increasing order of node identifier $i$.
According to the results, ILLS-MWOV outperforms all other methods in terms of MVs; as for TDs, the proposed method is slightly worse than IDLS and ILLS, while results are definitely better with respect to SWLS and EV, which shows the worst performance over such instances.
Notably, by considering simultaneously the results in terms of MVs and TDs, the ILLS-MWOV method dominates the EV and IWLS methodologies; as for the comparison with ILLS and IDLS, the proposed method results in a large improvement in terms of MVs with a degradation in terms of $TDs$ that is quite small with respect to ILLS and slightly more evident with respect to IDLS.
\color{black}

\section{Conclusions}
In this paper we develop a novel approach to reconstruct the ranking of a set of alternatives based on incomplete pairwise comparisons. Specifically, the proposed approach \textcolor{black}{blends cardinal and ordinal information}. This is done by considering two cascading optimization problems: first, we aim at finding an ordinal ranking that maximizes the accordance with the available information, then we seek a cardinal ranking via the logarithmic least-squares approach, with the additional constraint that the previously chosen ordinal ranking  is satisfied. Simulations show that the proposed approach is able to generate rankings that are not in contrast with the available information, while traditional approaches from the state of the art {\color{black} (i.e., \textcolor{black}{incomplete} logarithmic least squares, the \textcolor{black}{incomplete}  direct least squares, the weighted least squares, and approaches based on the eigenvectors) may fail. The proposed approach has been evaluated and compared, in terms of two common criteria, with respect to the other approaches from the state of the art. Our approach is able to guarantee the best results in terms of MVs (ordinal violations). Concerning the TDs criterion, our approach proposes solutions comparable to the classic \textcolor{black}{ILLS} approach. In this analysis, the gap between our solution and the \textcolor{black}{ILLS} solution is due to the presence of additional constraint necessary to preserve the ordinal ranking.
The simulation campaigns show that our approach proposes good solutions by varying both data inconsistency and graph density (i.e. the number of given ratios in $\textcolor{black}{\mathcal{A}}$).
}
Future work will aim at applying the proposed methodology to real-world situations, as well as considering a distributed computing setting.

\section*{Acknowledgements}
The authors would like to express their sincere gratitude to Dr.J\'anos F\"ul\"op for his valuable comments on the earlier versions of the paper.

\bibliographystyle{plain}

\section*{Biographies}

\noindent\textbf{Luca Faramondi } received the Laurea degree in Computer Science and Automation (2013) and the PhD degree on Computer Science and Automation (2017) from the University Roma Tre of Rome. He is currently PostDoc Fellow at Complex Systems \& Security Laboratory at the University Campus Bio- Medico of Rome. He is involved in several national and European projects about the Critical Infrastructure and Indoor Localization. His research interests include the identification of network vulnerabilities, cyber physical systems, and optimization at large.
\vspace{2mm}

\noindent\textbf{Gabriele Oliva} received the Laurea degree and the Ph.D in Computer Science and Automation Engineering in 2008 and 2012, respectively, both at University Roma Tre of Rome, Italy. He is currently assistant professor in Automatic Control at the University Campus Bio-Medico of Rome, Italy. His main research interests include distributed systems, distributed optimization, and applications of graph theory in technological and biological systems.
\vspace{2mm}

\noindent\textbf{S\'andor Boz\'oki} obtained his MSc degree in applied mathematics from E\"otv\"os Lorand
University, and PhD degree in economics from Corvinus University of Budapest, Hungary. He
is a senior research fellow at the Research Group of Operations Research and Decision
Systems, Laboratory on Engineering and Management Intelligence, Institute for Computer
Science and Control (SZTAKI). He is an associate professor at the Department of Operations
Research and Actuarial Sciences, Corvinus University of Budapest. His research interests
include multi-attribute decision making, pairwise comparison matrices, preference modelling,
global optimization and multivariate polynomial systems.
\end{document}